\documentclass[referee,envcountsect,]{svjour3} 
\smartqed
\usepackage{graphicx}
\journalname{Mathematical Programming}
\usepackage{amssymb,amsmath,color}
\usepackage{amssymb,mathrsfs}
\usepackage{hyperref}
\usepackage{enumerate}

\newcommand{\tcb}{\textcolor{blue}}
\newcommand{\tcr}{\textcolor{red}}

\newcommand{\R}{{\mathbb R}}
\newcommand{\N}{{\mathbb N}}
\DeclareMathOperator{\argmin}{argmin}
\newcommand{\interior}{{\rm int}\kern 0.06em}
\DeclareMathOperator*{\esssup}{ess\,sup}

\newcommand{\xk}{x_k}
\newcommand{\xkp}{x_{k+1}}
\newcommand{\xkm}{x_{k-1}}

\newcommand{\cH}{{\mathcal H}}

\newcommand{\demi}{\frac{1}{2}}
\newcommand{\ie}{{\it i.e.}\,\,}

\def\<{\langle}
\def\>{\rangle}

\newcommand{\dotp}[2]{\left\langle #1,\,#2 \right\rangle}
\newcommand{\norm}[1]{\left\|{#1}\right\|}
\newcommand{\pa}[1]{\left({#1}\right)}
\newcommand{\bpa}[1]{\Big({#1}\Big)}
\newcommand{\bra}[1]{\left[{#1}\right]}
\newcommand{\E}[1]{\mathbb{E}\bra{#1}}

\newcommand{\set}[1]{\left\{{#1}\right\}}

\newcommand{\prob}{\mathbb{P}}
\newcommand{\sigalg}{\mathcal{F}}
\newcommand{\events}{\Omega}
\newcommand{\prspace}{\pa{\events, \sigalg, \prob}}
\newcommand{\borel}{{\mathcal{B}}}
\newcommand{\filt}{{\mathcal{F}}}
\newcommand{\Filt}{{\mathscr{F}}}
\newcommand{\Rfilt}{{\mathcal{R}}}
\newcommand{\RFilt}{{\mathscr{R}}}
\newcommand{\Pas}{\pa{\prob\mbox{-a.s.}}}
\newcommand{\EX}[2]{\E{#1 \mid #2}}
\newcommand{\VarX}[2]{\mathrm{Var}\bra{#1 \mid #2}}
\newcommand{\distS}[1]{\mathrm{dist}(#1,S)}

\newcommand{\seq}[1]{\pa{#1}_{k \in \N}}

\newcommand{\eqdef}{:=}

\usepackage{ulem}

\usepackage{geometry}
 \usepackage{pict2e}

 \if
{
\usepackage[pageref]{backref}
\renewcommand*{\backrefalt}[4]{%
\ifcase #1 %
(Not cited)%
\or
(Cited on p.~#2)%
\else
(Cited on pp.~#2)%
\fi
}

}
\fi

\usepackage{geometry}
 \usepackage{pict2e}

\begin{document}
\title{The stochastic Ravine accelerated gradient method with general extrapolation coefficients}

\author{Hedy Attouch, Jalal Fadili and  Vyacheslav Kungurtsev}

\institute{
  Hedy Attouch \at IMAG CNRS UMR 5149,\\ Universit\'e Montpellier, \\ Place Eug\`ene Bataillon, 34095 Montpellier CEDEX 5, France. \\ hedy.attouch@umontpellier.fr
  \and
  Jalal Fadili \at GREYC CNRS UMR 6072 \\ Ecole Nationale Sup\'erieure d'Ing\'enieurs de Caen \\ 14050 Caen Cedex France. 
\\ Jalal.Fadili@greyc.ensicaen.fr
\and
Vyacheslav Kungurtsev
\at Department of Computer Science \\ Faculty of Electrical Engineering \\ Czech Technical University, \\ 12000 Prague, Czechia 
\\ vyacheslav.kungurtsev@fel.cvut.cz
}

\date{Received: date / Accepted: date\footnote{The insight and motivation for the study of the Ravine method, in light of its resemblance to Nesterov's, as well as many of the derivations in this paper, was the work of our beloved friend and colleague Hedy Attouch. As one of the final contributions of Hedy's long and illustrious career before his unfortunate recent departure, the other authors are fortunate to have worked with him on this topic, and hope that the polished manuscript is a valuable step in honoring his legacy.}}

\maketitle

\begin{abstract}
In a real Hilbert space domain setting, we study the convergence properties of the stochastic Ravine accelerated gradient method for convex  differentiable optimization. We consider the general form of this algorithm where the extrapolation coefficients can vary with each iteration, and where the evaluation of the gradient is subject to random errors. This general treatment models a breadth of practical algorithms and numerical implementations. 
We show that, under a proper tuning of the extrapolation parameters, and when the error variance associated with the gradient evaluations or the step-size sequences vanish sufficiently fast, the Ravine method provides fast convergence of the values both in expectation and almost surely. We also improve the convergence rates from $\mathcal{O}(\cdot)$ to $o(\cdot)$ in expectation and almost sure sense. Moreover, we show almost sure summability property of the gradients, which implies the fast convergence of the gradients towards zero. This property reflects the fact that the high-resolution ODE of the Ravine method includes a Hessian-driven damping term. When the space is also separable, our analysis allows also to establish almost sure weak convergence of the sequence of iterates provided by the algorithm. We finally specialize the analysis to consider different parameter choices, including vanishing and constant (heavy ball method with friction) damping parameter, and present a comprehensive landscape of the tradeoffs in speed and accuracy associated with these parameter choices and statistical properties on the sequence of errors in the gradient computations.
We provide a thorough discussion of the similarities and  differences with the Nesterov accelerated gradient which satisfies similar asymptotic convergence rates. 
\end{abstract}

\keywords{Ravine method \and Nesterov accelerated gradient method \and general extrapolation coefficient \and stochastic errors \and Hessian driven damping \and convergence rates \and Lyapunov analysis}

\subclass{37N40 \and 46N10 \and 49M30\and 65B99 \and 65K05 \and 65K10 \and 90B50 \and 90C25}


\section{Introduction}\label{sec:prel} 
Given a real Hilbert space $\cH$,  
our study concerns the fast numerical resolution of the convex minimization problem 
\smallskip
\begin{center}
$
(\mathcal P) \quad \min   \left\lbrace f(x): \; x\in \cH \right\rbrace,
$
\end{center}
\smallskip
by the  Ravine accelerated gradient method. We make the following standing assumptions:\smallskip
\begin{equation}\tag{$\rm{H}$}\label{assum:H}  
\begin{cases}
f: \cH \rightarrow  \R \; \mbox{ is differentiable,} \; \nabla f \mbox{ is } L-\mbox{Lipschitz continuous}, \; S= \argmin f \neq \emptyset. \vspace{1mm} \\
\seq{s_k} \mbox{ is a positive sequence with } s_kL \in ]0,1].
\end{cases}
\end{equation}

\medskip

The Ravine Accelerated Gradient algorithm (\ref{eq:RAGgammak} for short) generates iterates $\seq{y_k,w_k}$ satisfying
\begin{equation}\label{eq:RAGgammak}\tag{${\rm (RAG)}_{\gamma_k}$}
\begin{cases}
w_k=  y_k - s_k \nabla f(y_k)  \vspace{2mm} \\
y_{k+1} = w_k +  \gamma_k \left( w_k - w_{k-1}\right) . \rule{0pt}{3pt} \hspace{4cm}
\end{cases}
\end{equation}
Let us indicate the role of the different parameters involved in the above algorithm:
\begin{enumerate}[a)]
\item The positive parameter sequence $\seq{s_k}$ is the step-size sequence applied to the gradient based update.

\item The non-negative extrapolation coefficients $\seq{\gamma_k}$ are linked to the inertial character of the algorithm. They can be viewed as control parameters for optimization purposes. 

\item In order to inform about the practical performance of algorithms realizing this method in common capplications, 
we will analyze the convergence rates  when the gradient terms are calculated with stochastic errors. Formally, we consider $\nabla f(y_k) + e_k$ instead of $\nabla f(y_k)$ in \ref{eq:RAGgammak} where $e_k$ is a zero-mean stochastic error.
\end{enumerate}

The Ravine method is often confused with Nesterov's method \cite{Nest1,Nest2}, which is close in its formulation and its convergence properties. This justifies an in-depth study of the Ravine method and its comparison with Nesterov's method.

\subsection{Historical aspects}
The Ravine method was introduced by   Gelfand and  Tsetlin \cite{GT} in 1961. It is a first-order method which only uses evaluations of the gradient of $f$. It is closely related  with the Nesterov  accelerated  gradient method, with which it has often been confused.  
Recent research concerning the understanding of accelerated first-order optimization methods, seen as temporal discretized dynamic systems, has made it possible to clarify the link between the two methods; see the recent work of Attouch and Fadili \cite{AF}.

The Ravine method was introduced in \cite{GT} in the case of a fixed positive extrapolation coefficient $\gamma_k \equiv \gamma >0$. Recent research has shown the advantage of setting $\gamma_k = 1 - \frac{\alpha}{k}$, which, for $\alpha \geq 3$, provides asymptotic convergence rates similar to the accelerated gradient method of Nesterov.
The Ravine  method mimics the flow of water in the mountains which first flows rapidly downhill through small, steep ravines and then flows along the main river in the valley, hence its name.
 A geometric view of the Ravine Accelerated Gradient method  is given in Figure \ref{figRAG}. 

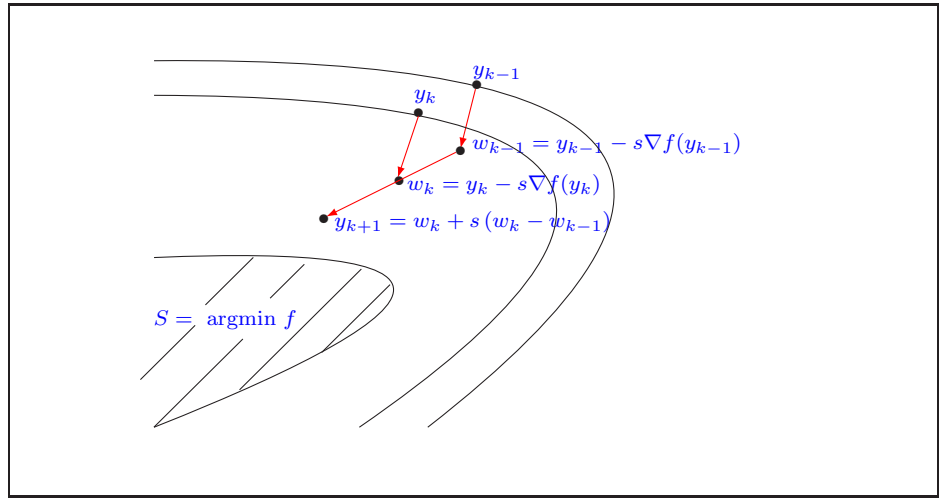
\begin{figure}[hbt!]
 \centering
 \fbox{\begin{minipage}{12cm}

\setlength{\unitlength}{9cm}
\begin{picture}(1,0.7)(-0.3,0.06)

\put(0.364,0.647){$\bullet$}
\tcr{
\put(0.37,0.648){\vector(-1,-4){0.022}}
}

\put(0.34,0.55){$\bullet$}

\put(0.278,0.605){$\bullet$}
\tcr{
\put(0.287,0.61){\vector(-1,-3){0.031}}
}

\put(0.25,0.505){$\bullet$}

\tcr{
 \put(0.341,0.555){\vector(-2,-1){0.19}}
 }
\put(0.025,0.205){\line(1,1){0.178}}
\put(0.145,0.26){\line(1,1){0.1}}
\put(-0.1,0.15){\line(1,1){0.13}}
\put(0.07,0.34){\line(1,1){0.05}}

\put(-0.12,0.22){\line(1,1){0.08}}
\put(-0.025,0.33){\line(1,1){0.07}}

\put(0.14,0.45){$\bullet$}

\tcb{
\put(0.366,0.669){$y_{k-1}$}
\put(0.28,0.633){$y_{k}$}
\put(0.365,0.56){$w_{k-1} =   y_{k-1} -s\nabla f (y_{k-1}) $}
\put(0.27,0.5){$w_{k} =  y_{k} -s \nabla f (y_{k}) $}
\put(0.164,0.445){$y_{k+1}=  w_k + s \left( w_k - w_{k-1}\right)$}
\put(-0.1,0.3){$S = \mbox{ \rm argmin } f$}
}

\qbezier(-0.1,0.15)(0.6,0.43)(-0.1,0.4)
\qbezier(0.3,0.15)(1.0,0.7)(-0.1,0.69)
\qbezier(0.2,0.15)(0.9,0.637)(-0.1,0.639)

\end{picture}
\end{minipage}}
\caption{ \textbf{\rm {\normalsize (RAG): Ravine Accelerated Gradient method }}}
\label{figRAG}
\end{figure}

%
%
%

\smallskip

The Ravine method was a precursor of the accelerated gradient methods.  
It  has long been ignored but has recently appeared at the forefront of current research in numerical optimization, see for example Polyak \cite{Polyak_0}, Attouch and Fadili \cite{AF}, Shi, Du, Jordan and Su \cite{SDJS}. It comes naturally into the picture when considering the optimized first-order methods for smooth convex minimization, see \cite{DT,KF,PPR}.
The Ravine and Nesterov acceleration methods are both based on the operations of extrapolation and gradient descent, but in a reverse order. 
Furthermore, up to a slight change in the extrapolation coefficients, the two algorithms are associated with the same equations, each of them describing the evolution of different variables, explaining how the two have been casually confused in some of the literature. 

The high resolution ODE of the two algorithms gives the same inertial dynamics with Hessian-driven damping, providing a mathematical basis to explain the similarity of their convergence properties. 
A significant difference is that in the Ravine method the discretized form of the Hessian driven term comes explicitly, while in the Nesterov method it comes implicitly by applying a Taylor formula. This results in different extensions of the two algorithms to the non-smooth case via the corresponding proximal gradient algorithms, an ongoing research topic. 
We will see that a careful adjustment of the extrapolation parameters $(\gamma_k)_{k\in \N}$ provides fast convergence properties of the Ravine algorithm resembling those of the Nesterov method.
Taking a general coefficient $\gamma_k$ gives a broad picture of the convergence properties of this class of algorithms. Moreover, it shows the flexibility of the method, the results being unchanged taking for example $\gamma_k=\frac{k}{k+ \alpha}$ instead of $1 - \frac{\alpha}{k}$, as one of the many variations of the method. 

\subsection{Inertial stochastic gradient descent algorithms}
Due to the importance of the subject in optimization, several works have been devoted to the study of perturbations in second-order dissipative inertial systems and in the corresponding first order algorithms (aka momentum methods). For deterministic perturbations, the subject was first considered for the case of a fixed viscous damping (aka heavy ball method with friction \cite{Polyak_1,Polyak_2}) in \cite{acz,HJ1}, then for the accelerated gradient method of Nesterov, and of the corresponding inertial dynamics with vanishing viscous damping, see \cite{AC2R-JOTA,ACPR,AD15,SLB,VSBV}. 

Stochastic gradient descent methods with inertia are widely used in applications and at the core of optimization subroutines in many applications such as machine learning. Such algorithms are the subject of an active research work to understand their convergence behaviour and were studied in several works, focusing exclusively on stochastic versions of Nesterov's method and the heavy ball method; see~\cite{Lin2015,Frostig2015,Jain2018,AR,AZ,Yang2016,Gadat2018,Loizou2020,Laborde2020,LanBook,Sebbouh2021,Driggs22}.  

\subsection{Contributions}
In this paper, in a real separable Hilbertian setting, we provide a unified analysis of the convergence properties of the Ravine method subject to noise in the gradient computation over a large class of the extrapolation sequence parameter settings beyond the standard ones for Nesterov's method. We will establish fast convergence rates in expectation and in almost sure sense on the objective values (both in $\mathcal{O}(\cdot)$ and $o(\cdot)$), on the gradient, and prove weak convergence of the sequence of iterates. This latter aspect is  overlooked by many existing works that focus exclusively on complexity estimates. These results will highlight the trade-off between the decrease of the error variance and fast convergence of the values and gradients. Our results cover some of those reviewed above as special cases for the Nesterov and heavy ball method extrapolation coefficients. In fact, even for these special cases, we complement the results of the literature with new ones. Moreover, we are not aware of any such a work for the Ravine method nor with general extrapolation coefficients.

\subsection{A model result}
Taking $\gamma_k = 1 -\frac{\alpha}{k}$ yields optimal convergence rate of the values and fast convergence of the gradients towards zero. Specifically, let the sequence $\seq{y_k}$ generated by the stochastic Ravine method with constant step-size
\begin{equation*}
\begin{cases}
w_k =  y_k - s (\nabla f(y_k) + e_k) \vspace{2mm} \\
y_{k+1} = w_k + \left(1 -\frac{\alpha}{k}\right)  \left(w_k - w_{k-1}\right), \rule{0pt}{5pt} \hspace{3cm}
\end{cases}
\end{equation*}
where $s \in ]0,1/L]$, $\seq{e_k}$ is a zero-mean stochastic noise. Let $\filt_k$ be the sub-$\sigma$-algebra generated by $y_0$ and $(w_{i})_{i \leq k-1}$. If $\alpha >3$, $\EX{e_k}{\filt_k}=0$ and $\sum_{k=1}^{+\infty} k\EX{\|e_k|\|^2}{\filt_k}^{1/2} < + \infty$ almost surely, then according to Theorem~\ref{thm:SRAG} and \ref{thm:SRAGgrad}, the following convergence properties hold: 
\begin{center}
\[
f(y_k)-  \min_{\cH} f = o \left(\frac{1}{k^2}\right) \quad \text{and} \quad \sum_k  k^2 \|\nabla f (y_k) \|^2 < +\infty \quad \text{almost surely} .
\]
\end{center}
In addition, if $\cH$ is also separable\footnote{Separability is crucial for proving almost sure weak convergence of the sequence of iterates.}, then the sequence $\seq{y_k}$ converges weakly almost surely to a random variable valued in $\argmin(f)$. Our results in Section~\ref{sec:special_cases} will be established for a much larger lass of the extrapolation sequence beyond $1-\alpha/k$. In particular, these results will emphasize the trade-off between the decrease of the error variance and fast convergence of the values and gradients.

\subsection{Contents} 
In Section~\ref{sec:Nest-Rav}, we start by making the link between the Ravine and the Nesterov method. This is instrumental because it makes it possible to transfer some known results of the Nesterov method. Then, in Section~\ref{sec:dyn} we show that the high resolution ODE of the Ravine method exhibits the damping governed by the Hessian. Section~\ref{sec:convergence} is devoted to the study of the convergence properties of the stochastic Ravine method, with as an important result the fast convergence in mean of the gradients towards zero. Section~\ref{sec:special_cases} contains illustrations of our results for various special choices of the extrapolation (inertial) sequence $\gamma_k$. Finally we conclude and draw some perspectives.

\section{Comparison of the Nesterov and Ravine methods}\label{sec:Nest-Rav}

Let us first recall some basic facts concerning the Nesterov method.

\subsection{Nesterov accelerated gradient method}
The Nesterov Accelerated Gradient (NAG for short) method with general extrapolation coefficients $\seq{\alpha_k}$, as studied in \cite{AC2R-JOTA}, reads
\begin{equation}\label{eq:NAGalphak}\tag{${\rm (NAG)}_{\alpha_k}$}
\begin{cases}
y_k =   x_{k} + \alpha_k ( x_{k}  - x_{k-1}) \\
\rule{0pt}{12pt}
x_{k+1} =  y_k- s_k \nabla f (y_k).
\end{cases}
\end{equation}
Its central role in optimization is due to the fact that a wise choice of the coefficients $\seq{\alpha_k}$ provides an optimal convergence rate of the values (in the worst case). 

Specifically, taking $\alpha_k= 1 -\frac{\alpha}{k}$ gives a scheme
%
which, for $\alpha \geq 3$, generates iterates $\seq{x_k}$ satisfying
\begin{equation}\label{Nest_conv_rate}
f(x_k) - \min_{\cH} f =\mathcal O \left( \frac{1}{k^2} \right)
\mbox{ as } k \to +\infty,
\end{equation}
and the fast convergence towards zero of the gradients (see \cite{AF})
\[
\sum_k  k^2 \| \nabla f (x_k) \|^2 < +\infty  .
\]
In addition, when $\alpha > 3$,
\begin{equation}\label{Nest_conv_rate_b}
f(x_k) - \min_{\cH} f =o \left( \frac{1}{k^2} \right)
\mbox{ as } k \to +\infty, \; \sum_k  k( f (x_{k})-  \min_{\cH} f )  < +\infty 
\end{equation}
and there is weak convergence of the iterates $\seq{x_k}$ to optimal solutions,
 see \cite{ACPR,AC1,AC2,AP,SBC}.

\subsection{Passing from Nesterov method to Ravine method and vice versa}
 
To avoid confusion between the two algorithms \ref{eq:RAGgammak} and \ref{eq:NAGalphak}, we use the  subscript $\gamma_k$ for the extrapolation coefficient in the Ravine method, and $\alpha_k$ for the extrapolation coefficient in the Nesterov method. 
A remarkable fact is that the variable $y_k$ which enters the definition of \ref{eq:NAGalphak} follows the \ref{eq:RAGgammak} algorithm, with $\gamma_k = \alpha_{k+1}$. This generalizes the observation already made in \cite{AF} for the specific choice $\alpha_k=1 -\frac{\alpha}{k}$.
Although this is an elementary result, we give a detailed account of it in the following theorem, due to its importance. 

\begin{theorem}\label{Nest_Ravine}
\begin{enumerate}[(i)]
\item Let $\seq{x_k}$ be the sequence generated by the Nesterov algorithm \ref{eq:NAGalphak}.
Then the associated sequence $\seq{y_k}$ 
also follows the equations of the Ravine algorithm \ref{eq:RAGgammak} with $\gamma_k = \alpha_{k+1}$

\item Conversely, if $\seq{y_k}$ is the sequence associated to the Ravine method \ref{eq:RAGgammak},
then the sequence $\seq{x_k}$ defined by $x_{k+1} \eqdef y_k - s_k \nabla f (y_k)$ follows the Nesterov algorithm \ref{eq:NAGalphak} with $\alpha_k = \gamma_{k-1}$.
\end{enumerate}
\end{theorem} 
\begin{proof}
\begin{enumerate}[(i)]
\item Suppose that $\seq{x_k}$  follows  \ref{eq:NAGalphak}. According to the definition of $y_k$
\begin{eqnarray*}
y_{k+1}&=&   x_{k+1} +\alpha_{k+1} ( x_{k+1}  - x_{k})\\
&=&  y_k- s_k\nabla f (y_k) + \alpha_{k+1}  \Big( y_k- s_k\nabla f (y_k)  - ( y_{k-1}- s_{k-1}\nabla f (y_{k-1}) )\Big).
\end{eqnarray*}
Set $w_k \eqdef y_k - s_k\nabla f (y_k)$ (which is nothing but $x_{k+1}$). We obtain that $\seq{y_k}$ follows ${\rm(RAG)}_{\alpha_{k+1}}$,  \ie
\begin{equation*}{\rm (RAG)}_{\alpha_{k+1}}
\begin{cases}
w_k = y_k - s_k\nabla f (y_k) \\ \rule{0pt}{15pt}
y_{k+1} = w_k + \alpha_{k+1} \left( w_k - w_{k-1}\right).
\end{cases}
\end{equation*}
\item Conversely, suppose that $\seq{y_k}$ follows the Ravine method
\ref{eq:RAGgammak}. According to the definition of $y_{k+1}$ and $w_k$, we have 
\[
y_{k+1} = y_k - s_k\nabla f (y_k) + \gamma_k \Big(y_k - s_k\nabla f (y_k) -( y_{k-1}- s_{k-1}\nabla f (y_{k-1}))\Big).
\]
By definition of $x_{k+1} =  y_k - s_k \nabla f (y_k)$, we deduce that 
\[
y_{k+1} = x_{k+1} + \gamma_k \left(x_{k+1} - x_{k}\right).
\]
Equivalently
\[
y_{k} = x_{k} + \gamma_{k-1} \left(x_{k} - x_{k-1}\right).
\]
Putting together the above relations and the definition of $x_{k+1}$, we obtain that $\seq{x_k}$
follows ${\rm (NAG)}_{\gamma_{k-1}}$, \ie
\begin{equation*}{\rm (NAG)}_{\gamma_{k-1}} 
\begin{cases}
y_k = x_{k} + \gamma_{k-1} ( x_{k}  - x_{k-1}) \\
\rule{0pt}{12pt}
x_{k+1} = y_k - s_k \nabla f (y_k).
\end{cases}
\end{equation*}
This completes the proof. \qed
\end{enumerate}
\end{proof}

Though the two methods are intimately linked as we have just seen, it is only recent advances in the dynamical system interpretation of the two methods that revealed their close relationship and also their differences. This is  explained in the next section, where we consider the case of the Ravine method with general extrapolation coefficients, hence generalizing the work of \cite{AF} beyond the case $\alpha_k = 1-\alpha/k$.

\section{The Ravine method from a dynamic perspective}\label{sec:dyn}
We  consider   the high resolution ODE of the  Ravine method, and show that it exhibits damping governed by the Hessian. This will explain the fast convergence towards zero of the gradients satisfied by the Ravine method, a claim that we will prove in the next section.

\subsection{Dynamic tuning of the extrapolation coefficients}
Let us first explain how to tune the extrapolation coefficients in the Ravine method, in order to obtain a dynamic interpretation of the algorithm.
Critical to the understanding is the link between the Ravine method and the Nesterov method, as explained in Section~\ref{sec:Nest-Rav}, and the  dynamic interpretation of the Nesterov method, due to Su, Boyd and 
Cand\`es \cite{SBC}.
Consider  the inertial gradient system
\begin{equation*}\label{eq:igs}\tag{${\rm(IGS)}_{\gamma}$}
\ddot{x}(t) + \gamma(t) \dot{x}(t) + \nabla f (x(t))=0,
\end{equation*}
which involves a general viscous damping coefficient $\gamma(\cdot)$.
The implicit time discretization of \ref{eq:igs}, with time step-size $h>0$, $x_k= x(\tau_k)$, and $\tau_k=kh$ \footnote{We take the $\tau_k$ notation  instead of the usual  $t_k$, because $t_k$ will be used with a different meaning, and it is used extensively in the paper.}, gives
\[
\frac{ x_{k+1} - 2x_{k}+ x_{k-1} }{h^2} +   \gamma (kh) \frac{x_{k} - x_{k-1}}{h} + \nabla f( x_{k+1}) = 0.
\] 
Let $s \eqdef h^2$. After multiplication by $s$, we obtain
\begin{equation}
(x_{k+1} -x_{k})- (x_{k} -x_{k-1}) + h\gamma (kh)(x_{k} - x_{k-1}) + s \nabla f( x_{k+1})=0. \label{basic-eq-1}
\end{equation}
Equivalently
\begin{equation}
x_{k+1}  + s \nabla f( x_{k+1})= x_{k}+  \left( 1- h\gamma (kh)\right)(x_{k} -x_{k-1}) , \label{basic-eq-2}
\end{equation}
which gives
\begin{equation}
x_{k+1} = {\rm prox}_{s f} \left( x_{k}+  \left( 1- h\gamma (kh)\right)(x_{k} -x_{k-1})\right) . \label{basic-eq-3}
\end{equation}
We obtain the inertial proximal algorithm 
\begin{eqnarray*}
\begin{cases}
y_k=  \displaystyle{ x_{k}+  \left( 1- h\gamma (kh)\right)(x_{k} -x_{k-1})} \hspace{1cm} \vspace{1mm} \\
x_{k+1} = {\rm prox}_{s f} \left( y_{k} \right).
\end{cases}
\end{eqnarray*}
Following the general procedure described in \cite{AF}, which consists in  replacing the proximal step by a gradient step, we obtain \ref{eq:NAGalphak}
with $\alpha_k=  1- h\gamma (kh)$.
Taking $\gamma(t) =\frac{\alpha}{t}$, we obtain \ref{eq:NAGalphak} with $\alpha_k=1 -\frac{\alpha}{k}$, 
which provides fast convergence results.
Observe that Algorithm~\ref{eq:NAGalphak} makes sense for any arbitrarily given sequence of positive numbers $\seq{\alpha_k}$. But for this algorithm to be directly connected by temporal discretization to the continuous dynamic \ref{eq:igs}, it is necessary to take $\alpha_k=  1- h\gamma (kh)$.
Note that the case $\gamma(t) =\frac{\alpha}{t}$ is special, since due to the homogeneity property of $\gamma(\cdot)$, in this case $\alpha_k$ does not depend on $h$.

\smallskip

Let us now use  the relations established in Section~\ref{sec:Nest-Rav} between the Nesterov and the Ravine methods. 
Since $\seq{x_k}$  satisfies \ref{eq:NAGalphak} with  
$\alpha_k=  1- h\gamma (kh)$, we have that  the associated sequence $\seq{y_k}$ follows \ref{eq:RAGgammak} with $\gamma_k = \alpha_{k+1} = 1- h\gamma ((k+1)h)$.
%

\subsection{High resolution ODE of the Ravine method}
Let us now proceed with the high resolution ODE of the Ravine method \ref{eq:RAGgammak}.
The idea is  not to let $h \to 0$,  but to take into account  the  terms of order $h=\sqrt{s}$ in the asymptotic expansion, and to neglect the term of order $h^2=s$.
The high resolution  method is  extensively used in fluid mechanics, where physical phenomena occur at multiple scales.
Indeed, by following an approach  similar to that
developed by Shi, Du,  Jordan, and  Su in \cite{SDJS}, and Attouch and Fadili in \cite{AF}, we are going to show that  the Hessian-driven damping  appears in the associated continuous inertial equation. Let us make this precise in the following result.

\begin{theorem}\label{highresolution}
The high resolution ODE with temporal step size $h= \sqrt{s}$ of the Ravine method \ref{eq:RAGgammak} with $\gamma_k = h\gamma ((k+1)h)$
%
gives the inertial dynamic with Hessian driven damping
\begin{equation}\label{edo101}
\ddot{y}(t) + \gamma(t) \left(1+ \frac{\sqrt{s}}{2} \gamma (t)\right)  \dot{y}(t)  
+ \sqrt{s}\nabla^2 f (y(t)) \dot{y}(t)+ \left(1+ \frac{\sqrt{s}}{2} \gamma (t)\right) \nabla f (y(t))  =0.
\end{equation}
\end{theorem}

\begin{proof}
Set $\gamma_k = 1- h\gamma ((k+1)h)$.
By definition of the Ravine method
$$
y_{k+1} = y_k- s\nabla f (y_k) + \gamma_k \Big (y_k- s\nabla f (y_k))- ( y_{k-1}- s\nabla f (y_{k-1}))   \Big) .
$$
Equivalently
$$
(y_{k+1}-y_k)- (y_k- y_{k-1}) +(1- \gamma_k) (y_k- y_{k-1})+     s\nabla f (y_k) + s\gamma_k \Big (\nabla f (y_k)-\nabla f (y_{k-1})   \Big)=0 .
$$
After dividing by $s=h^2$,  we obtain
\begin{equation}\label{dyn-Rav-1c}
\frac{ y_{k+1} - 2y_{k}+ y_{k-1} }{h^2}  + (1- \gamma_k) \frac{y_k- y_{k-1}}{h^2} +\nabla f (y_k)
+ \gamma_k( \nabla f (y_k) -\nabla f (y_{k-1} )) =0.
\end{equation}
Notice then that 
$$
\frac{y_k- y_{k-1}}{h^2} = \frac{y_{k+1} -y_k}{h^2} - \frac{ y_{k+1} - 2y_{k}+ y_{k-1} }{h^2}. 
$$
So, \eqref{dyn-Rav-1c} can be  formulated equivalently as follows
\begin{equation}\label{dyn-Rav-1cc}
\gamma_k \frac{ y_{k+1} - 2y_{k}+ y_{k-1} }{h^2}  + (1- \gamma_k) \frac{y_{k+1}- y_{k}}{h^2} +\nabla f (y_k)
+ \gamma_k( \nabla f (y_k) -\nabla f (y_{k-1} )) =0.
\end{equation}
After dividing  by $\gamma_k$,  we get 
\begin{equation}\label{dyn-Rav-1cd}
 \frac{ y_{k+1} - 2y_{k}+ y_{k-1} }{h^2}  +
  \frac{1- \gamma_k}{h\gamma_k} \frac{y_{k+1}- y_{k}}{h} 
  + \frac{1}{\gamma_k}\nabla f (y_k)
+ ( \nabla f (y_k) -\nabla f (y_{k-1} )) =0.
\end{equation}
\noindent Building on \eqref{dyn-Rav-1cd}, we use Taylor expansions taken at a higher order (here,  order four) than for the low resolution ODE. 
For each $k\in\N$, set $\tau_k=(k +c)h $, where $c$ is a real parameter that will be adjusted further.
Assume that  $y_k = Y(\tau_k)$ for some  smooth curve  $\tau \mapsto Y(\tau)$ defined for $\tau\geq t_0 >0$. Performing a Taylor expansion in powers of $h$, when $h$ is close  to zero, of the different quantities involved in \eqref{dyn-Rav-1cd}, we obtain
\begin{eqnarray}
y_{k+1} &=& Y(\tau_{k+1})= Y(\tau_k) + h \dot{Y}(\tau_k)+ \demi h^2 \ddot{Y}(\tau_k) + \frac{1}{6}h^3 \dddot{Y}(\tau_k)+ 
\mathcal O (h^4) \label{taylor1c}\\
y_{k-1} &=& Y(\tau_{k-1})= Y(\tau_k) - h \dot{Y}(\tau_k)+ \demi h^2 \ddot{Y}(\tau_k)  - \frac{1}{6}h^3 \dddot{Y}(\tau_k)+ 
\mathcal O (h^4) . \label{taylor2c}
\end{eqnarray}
By adding \eqref{taylor1c} and \eqref{taylor2c}  we obtain 
$$
\frac{ y_{k+1} - 2y_{k}+ y_{k-1} }{h^2} =   \ddot{Y}(\tau_k) + \mathcal O (h^2).
$$
Moreover, \eqref{taylor1c} gives
$$
\frac{y_{k+1} - y_{k}}{h} =   \dot{Y}(\tau_k) + \demi h \ddot{Y}(\tau_k)  + \mathcal O (h^2).
$$
\noindent  By Taylor expansion of $\nabla f$  we have
$$ \nabla f (y_k) -\nabla f (y_{k-1})=   
 h\nabla^2 f (Y(\tau_k)) \dot{Y}(\tau_k) + \mathcal O \left(h^2\right).
$$
Plugging all of the above results into \eqref{dyn-Rav-1cd}, we obtain
\begin{multline*}
[\ddot{Y}(\tau_k)+\mathcal O (h^2)  ] +  \frac{1- \gamma_k}{h\gamma_k} \left[\dot{Y}(\tau_k)+ \demi h \ddot{Y}(\tau_k) 
 + \mathcal O (h^2) \right] \\
 +  \frac{1}{\gamma_k} \nabla f (Y(\tau_k)) 
+\left[ h\nabla^2 f (Y(\tau_k)) \dot{Y}(\tau_k) + \mathcal O \left( h^2\right)\right] =0.
\end{multline*}
Multiplying  by $\frac{h\gamma_k}{1- \gamma_k}$, we obtain in an equivalent way 
$$
\frac{h\gamma_k}{1- \gamma_k} \ddot{Y}(\tau_k) +  \dot{Y}(\tau_k) + \demi h \ddot{Y}(\tau_k) + \frac{h}{1-\gamma_k}\nabla f (Y(\tau_k)) 
+\frac{h^2\gamma_k}{1- \gamma_k}\nabla^2 f (Y(\tau_k)) \dot{Y}(\tau_k) +\mathcal O (h^3)  =0.
$$
Afterreduction of the terms involving $\ddot{Y}(t_k)$,  we obtain 
\begin{eqnarray*}
&&\frac{h(1+ \gamma_k)}{2(1-\gamma_k)}\ddot{Y}(\tau_k)+ \dot{Y}(\tau_k) +  \frac{h}{1-\gamma_k}\nabla f (Y(\tau_k)) 
+\frac{h^2\gamma_k}{1- \gamma_k}\nabla^2 f (Y(\tau_k)) \dot{Y}(\tau_k) +\mathcal O (h^3)  =0.
\end{eqnarray*}
Multiplication by $\frac{2(1-\gamma_k)}{h(1+ \gamma_k)}$ then yields
\begin{eqnarray}
&&\ddot{Y}(\tau_k)+ \frac{2(1-\gamma_k)}{h(1+ \gamma_k)} \dot{Y}(\tau_k) +  \frac{2}{1+\gamma_k}\nabla f (Y(\tau_k)) 
+\frac{2h \gamma_k}{1+\gamma_k}\nabla^2 f (Y(\tau_k)) \dot{Y}(\tau_k) +\mathcal O (h^2)  =0. \label{eq:discrete_28_12_2021}
\end{eqnarray}
According to $\gamma_k = 1- h\gamma ((k+1)h)$, and $\tau_k= (k+1)h$, we obtain 
\begin{equation*}
\ddot{Y}(\tau_k) + \frac{\gamma (\tau_k)}{1-\frac{h}{2} \gamma (\tau_k)} \dot{Y}(\tau_k) + 
\frac{1}{1-\frac{h}{2} \gamma (\tau_k)} \nabla f (Y(\tau_k)) 
+ h \frac{1- h\gamma (\tau_k)}{1- \frac{h}{2} \gamma (\tau_k)} \nabla^2 f (Y(\tau_k)) \dot{Y}(\tau_k) +\mathcal O (h^2)  =0.
\end{equation*}
By neglecting the term of order $s=h^2$, and keeping the terms of order 
$h= \sqrt{s}$, we obtain the inertial dynamic with
Hessian driven damping
$$
\ddot{Y}(t) + \gamma(t) \left(1+ \frac{\sqrt{s}}{2} \gamma (t)\right)  \dot{Y}(t)  
+ \sqrt{s}\nabla^2 f (Y(t)) \dot{Y}(t)+ \left(1+ \frac{\sqrt{s}}{2} \gamma (t)\right) \nabla f (Y(t))  =0.
$$
This completes the proof.  \qed
\end{proof}

\begin{remark}{
The high resolution ODE of the Ravine method exhibits Hessian driven damping. In addition, it incorporates a gradient correcting term weighted with a coefficient of $\left(1+ \frac{\sqrt{s}}{2} \gamma (t)\right)$. This is in accordance with \cite{AF} and \cite{SDJS}.
Surprisingly, there is also a correction which appears in the viscosity term, the coefficient  $\left(1+ \frac{\sqrt{s}}{2} \gamma (t)\right)  $ in front of the velocity.
Indeed as we already observed,  the Nesterov case is very specific. When $\gamma (t)=\frac{\alpha}{t}$,  we have 
$s = 1- h\gamma ((k+1)h)= 1- \frac{\alpha}{k+1}$.
Returning to \eqref{eq:discrete_28_12_2021}, we have}
$$
\frac{2(1-s)}{h(1+ s)} = \frac{\alpha}{h(k+1 -\frac{\alpha}{2})}.
$$
{\rm Taking $\tau_k = h(k+1 -\frac{\alpha}{2})$
gives $\gamma (\cdot )$ as the viscosity coefficient of the 
limit equation.}
\end{remark}


\section{Convergence properties of the stochastic Ravine method}\label{sec:convergence}
In this section, we analyze the convergence properties of the Ravine method with stochastic errors in the eveluation of the gradients. We first examine the fast convergence of the values and the convergence of iterates, then we show the fast convergence of the gradients towards zero. This section considers the algorithmic and stochastic version of the results obtained by the authors for the corresponding continuous dynamical systems with deterministic errors \cite{AFK}.

\subsection{Values convergence rates and convergence of the iterates}
We first start by proving the results for the Nesterov method before transferring them to the Ravine method thanks to Theorem~\ref{Nest_Ravine}. In \cite{AC2R-JOTA}, the Nesterov accelerated gradient method with a general extrapolation coefficient $\alpha_k$ and deterministic terms was studied. Here, we consider a stochastic version which reads for $k \geq 1$ 
\begin{equation}\tag{${\rm(SNAG)}_{\alpha_k}$}\label{eq:NAGalphakstoch}
\begin{cases}
y_k=x_k+\alpha_k(x_k-x_{k-1})\\\rule{0pt}{15pt}
x_{k+1}= y_k - s_k (\nabla f(y_k) + e_k)
\end{cases}
\end{equation}
where $s_k \in ]0,1/L]$ is a sequence of step-sizes, $\seq{e_k}$ is a sequence of $\cH$-valued random variables. \ref{eq:NAGalphakstoch} is initialized with $x_0=x_1$, where $x_0$ a $\cH$-valued, squared integrable random variable.

Taking the objective function $f\equiv 0$ and $e_k \equiv 0$ in \ref{eq:NAGalphakstoch} already reveals insights for choosing the best parameters. In this case, the algorithm \ref{eq:NAGalphakstoch} becomes $\xkp-\xk-\alpha_k(\xk-\xkm)=0.$
This implies that for every $k\geq 1$,
\[
x_k=x_1+\left(\sum_{i=1}^{k-1}\prod_{j=1}^i \alpha_j\right)(x_1-x_0).
\]
Therefore, $\seq{x_k}$ converges if and only if $\sum_{i=1}^{+\infty}\prod_{j=1}^i \alpha_j <+\infty$.
We are naturally led to introduce the sequence $\seq{t_k}$ defined by
\begin{equation}\label{def:tk}
t_k \eqdef 1+\sum_{i=k}^{+\infty}\prod_{j=k}^i \alpha_j.
\end{equation}
The above formula may seem complicated at a first glance. In fact,  the inverse transform, which makes it possible to pass from $ t_k $ to 
 $ \alpha_k $ has the following, simpler form 
\begin{equation}\label{inverse_tk}
\alpha_k = \frac{t_k -1}{t_{k+1}}.
\end{equation}
Formula \eqref{inverse_tk}  will ease the path of the analysis and we shall make regular use of it
in the sequel.
%
%

\medskip

From now on, we denote by $\prspace$ a probability space. We assume that $\cH$ is a real separable Hilbert space endowed with its Borel $\sigma$-algebra, $\borel\pa{\cH}$. We denote a filtration on $\prspace$ by $\Filt \eqdef \seq{\filt_k}$ where $\filt_k$ is a sub-$\sigma$-algebra satisfying, for each $k \in \N$, $\filt_k\subset \filt_{k+1}\subset\sigalg$. Furthermore, given a set of random variables $\set{a_0,\ldots,a_k}$ we denote by $\sigma\pa{a_0,\ldots,a_k}$ the $\sigma$-algebra generated by $a_0,\ldots,a_k$. Finally, a statement $\pa{P}$ is said to hold $\Pas$ if 
\[
\prob\pa{\set{\omega\in\events: \pa{P}\mbox{ holds}}}=1.
\]
Using the above notation, we denote the canonical filtration associated to the iterates of algorithm \ref{eq:NAGalphakstoch} as $\Filt$ with, for all $k\in\N$,
\[
\filt_k \eqdef \sigma\pa{x_0,\ldots,x_k}
\]
such that all iterates up to $x_k$ are completely determined by $\filt_k$. 

For the remainder of the paper, all equalities and inequalities involving random quantities should be understood as holding $\Pas$ even if it is not explicitly written. 

\begin{definition}
Given a filtration $\Filt$, we denote by $\ell_+\pa{\Filt}$ the set of sequences of $[0,+\infty[$-valued random variables $\seq{a_k}$ such that, for each $k\in\N$, $a_k$ is $\filt_k$-measurable. Then, for $p \in ]0,+\infty[$, we also define the following set of $p$-summable random variables,
\[
\ell^p_+\pa{\Filt} \eqdef \set{\seq{a_k}\in\ell_+\pa{\Filt}:\sum\limits_{k\in\N} a_k^p <+\infty \; \Pas}.
\]
The set of non-negative $p$-summable (deterministic) sequences is denoted $\ell^p_+$.
\end{definition}

The following theorem is a generalization of \cite[Theorems~3.1, 3.2 and 3.4]{AC2R-JOTA} to the stochastic setting.
\begin{theorem}\label{thm:SNAG}
Assume that \eqref{assum:H} holds and the sequence $\seq{\alpha_k}$ satisfies 
\begin{align}
\forall k \geq 1,  \quad  \displaystyle{\sum_{i=k}^{+\infty} \prod_{j=k}^{i} \alpha_j <+\infty}, \tag{$K_0$}\label{eq:K0} \\
\forall k \geq 1,  \quad  t_{k+1}^2 -t_{k}^2 \leq t_{k+1} . \tag{$K_1$} \label{eq:K1}
\end{align}
Consider the algorithm \ref{eq:NAGalphakstoch} where $s_k \in ]0,1/L]$ is a non-increasing sequence and $\seq{e_k}$ is a sequence of stochastic errors such that
\begin{equation}\label{eq:K2}\tag{$K_2$}
\EX{e_k}{\filt_k} = 0 \; \Pas \quad \text{and} \quad \seq{s_k t_k \sigma_k} \in \ell^2_+(\Filt),
\end{equation}
where $\sigma_k^2 \eqdef \EX{\norm{e_k}^2}{\filt_k}$.
Then,
\begin{enumerate}[(i)] 
\item \label{thm:Thm-Nest-stoch-claim1}
we have the following rate of convergence in almost sure and mean sense:
\[
f(x_k)-  \min f = \mathcal O \pa{\frac{1}{s_k t_{k}^2}} \; \Pas , 
\] 
and
\[
\E{f(x_k)-  \min f} \leq \frac{s_1t_1^2\E{f(x_0)-\min f} + \frac{1}{2}\E{\distS{x_0}^2} + 4\sum_{i=1}^{+\infty} s_i^2 t_i^2\E{\norm{e_i}^2}}{s_k t_k^2} .
\]

\item \label{thm:Thm-Nest-stoch-claim2}
Assume in addition that, for $m\in [0,1[$, 
\begin{equation}\tag{$K_1^+$}\label{eq:K1+} 
t_{k+1}^2-t_k^2\leq m\,t_{k+1}\quad \mbox{ for every } k\geq 1 ,
\end{equation}
then 
\[
\sum_{k \in \N} s_k t_{k+1}(f(x_k)-\min f) < +\infty \quad \text{and} \quad \sum_{k \in \N} t_k \norm{x_k-x_{k-1}}^2 < +\infty \; \Pas .
\]
If moreover $\sum_{k \in \N} \frac{t_{k+1}}{t_k^2} = +\infty$, then
\[
f(x_k)-\min f = o\pa{\frac{1}{s_k t_k^2}} \quad \text{and} \quad \norm{x_k-x_{k-1}} = o\pa{\frac{1}{t_k}} \; \Pas . 
\]

\item \label{thm:Thm-Nest-stoch-claim3}
If $\alpha_k\in [0,1]$ for every $k\geq 1$, $\inf_k s_k > 0$, \eqref{eq:K1+} holds and \eqref{eq:K2} is strengthened to
\begin{equation}\label{eq:K2+}\tag{$K_2^+$}
\EX{e_k}{\filt_k} = 0 \; \Pas \quad \text{and} \quad \seq{s_k t_k\sigma_k} \in \ell^1_+(\Filt),
\end{equation}
then the sequence $\seq{x_k}$ converges weakly $\Pas$ to an $\argmin(f)$-valued random variable.
\end{enumerate}
\end{theorem}

Before delving into the proof, some remarks are in order.
\begin{remark}
From claim \eqref{thm:Thm-Nest-stoch-claim1}, we have, for $s_k$ constant and bounded away from $0$, convergence at the rate $O(1/k^2)$ in the objective if $\seq{t_k\sigma_k} \in \ell^2_+(\Filt)$. If just $\seq{\sigma_k} \in \ell^2_+(\Filt)$, then the step-size must anneal at the rate $s_k \sim 1/t_k$ for an objective value convergence rate $O(1/t_k)$. 

Now consider non-vanishing noise with bounded variance (i.e. $\lim\sup\sigma_k>0,\, \mathbb{P}-a.s.$ and $\mathbb{E}[\sigma^2_k]-\mathbb{E}[\sigma_k]^2\le B,\, 0<B<\infty$). For the choice of $t_k=(k-1)/(\alpha-1)$, setting the step-size to be $s_k=1/k^{1+\delta}$, with $\delta > 0$, results in convergence with a rate is $O(1/k^{\delta})$. If $s_k=1/k$ and the noise does not asymptotically vanish (a.s.), convergence can only be ensured to a noise dominated region. On the other hand, if $t_k=(k^{1+\delta}-1)/(\alpha-1)$ with $\delta<0$, then $s_k=1/k$ achieves a convergence rate of $O(1/k^{\delta})$ if there is vanishing noise. Continuing, we see that the $O(1/k^2)$ rate is achieved for vanishing noise and $s_k=1/k^{(2-\delta)}$.

The last statement of claim \eqref{thm:Thm-Nest-stoch-claim2} can be modified to get the same rate as in the deterministic case in \cite[Theorem~3.4]{AC2R-JOTA} but only at the price of a stronger summability assumption on the noise.
\end{remark}

\begin{proof}
Our proof is based on a (stochastic) Lyapunov analysis with appropriately chosen energy functionals. 

\noindent\eqref{thm:Thm-Nest-stoch-claim1} 
Denote $f_k(x) \eqdef f(x) + \dotp{e_k}{x}$ and recall $S = \argmin(f)$. Define the sequence
\begin{eqnarray*}
V_k \eqdef  s_k t_k^2( f(x_k)-  f( x^\star) ) +\frac{1}{2}\distS{z_k}^2  & \text{and} &
z_k \eqdef  x_{k-1} + t_k\pa{x_{k} - x_{k-1}} .
\end{eqnarray*}
Since $f$ is convex and $L$-smooth, so is $f_k$. Let us apply \eqref{eq:extdesclem} in Lemma~\ref{ext_descent_lemma} on $f_k$ successively at $y=y_k$ and $x= x_k$, then at $y=y_k$ and $x= x^\star \in S$. We get
\begin{align}
&f_k(x_{k+1}) \leq f_k(x_k) + \dotp{\nabla f_k(y_k)}{y_k-x_k} - \frac{s_k}{2}\norm{\nabla f_k(y_k)}^2 \label{eq.combinaison_eq_prec_1a}\\
&f_k(x_{k+1}) \leq f_k(x^\star) + \dotp{\nabla f_k(y_k)}{y_k-x^\star} - \frac{s_k}{2} \norm{\nabla f_k (y_k)}^2 . \label{eq.combinaison_eq_prec_1b}
\end{align}
Multiplying \eqref{eq.combinaison_eq_prec_1a} by $t_{k+1}-1$ (which is non-negative by definition), then adding the \eqref{eq.combinaison_eq_prec_1b}, we derive that
\begin{multline}\label{eq.combinaison_eq_prec_2}
t_{k+1}f_k(x_{k+1}) 
\leq (t_{k+1}-1)f_k(x_k) + f_k(x^\star) +  \dotp{\nabla f_k(y_k)}{(t_{k+1}-1)(y_k-x_k) + y_k - x^\star} \\
- \frac{s_k}{2}t_{k+1} \norm{\nabla f_k (y_k)}^2 .
\end{multline}
It is immediate to see, using \eqref{inverse_tk} and the definitions of $y_k$ and $z_k$, that
\begin{align*}
(t_{k+1}-1)(y_k-x_k) + y_k = x_k + t_{k+1}(y_k-x_k) 
&= x_{k-1} + (1+t_{k+1}\alpha_k)(x_k-x_{k-1}) \\
&= x_{k-1} + t_k(x_k-x_{k-1}) = z_k .
\end{align*}
Inserting this into \eqref{eq.combinaison_eq_prec_2} and rearranging, we get
\begin{align}\label{eq.combinaison_eq_prec_3}
t_{k+1}(f_k(x_{k+1}) - f_k(x^\star))
\leq (t_{k+1}-1)(f_k(x_k) - f_k(x^\star)) + \left\langle \nabla f_k(y_k), z_k - x^\star \right\rangle - \frac{s_k}{2}t_{k+1} \norm{\nabla f_k (y_k)}^2 .
\end{align}
Straightforward computation, using again \eqref{inverse_tk} and the definition of $y_k$ and $z_k$, can yield the expression,
\begin{equation}\label{eq:zkrec}
z_{k+1} - z_k = -s_k t_{k+1} \nabla f_k(y_k) .
\end{equation}
Thus
\[
\norm{z_{k+1} - x^\star}^2 = \norm{z_k - x^\star}^2 - 2s_k t_{k+1} \dotp{\nabla f_k(y_k)}{z_k-x^\star} + s_k^2t_{k+1}^2\norm{\nabla f_k(y_k)}^2 .
\]
Dividing this by $2$ and adding to \eqref{eq.combinaison_eq_prec_3}, after multiplying the latter by $s_k t_{k+1}$, cancels all terms containing $\nabla f(y_k)$ and we arrive at
\begin{align}\label{eq.combinaison_eq_prec_4}
s_k t_{k+1}^2(f_k(x_{k+1}) - f_k(x^\star)) + \frac{1}{2}\norm{z_{k+1}-x^\star}^2
\leq s_k t_{k+1}(t_{k+1}-1)(f_k(x_k) - f_k(x^\star)) + \frac{1}{2}\norm{z_k-x^\star}^2 .
\end{align}
Let us take $x^\star$ as the closest point to $z_k$ in $S$. Thus \eqref{eq.combinaison_eq_prec_4} is equivalent to
\begin{align}\label{eq.combinaison_eq_prec_5}
s_kt_{k+1}^2(f_k(x_{k+1}) - f_k(x^\star)) + \frac{1}{2}\distS{z_{k+1}}^2
\leq s_kt_{k+1}(t_{k+1}-1)(f_k(x_k) - f_k(x^\star)) + \frac{1}{2}\distS{z_k}^2 .
\end{align}
Let us now isolate the error terms. Inequality \eqref{eq.combinaison_eq_prec_5} is then equivalent to
\begin{multline}\label{eq.combinaison_eq_prec_6}
s_kt_{k+1}^2(f(x_{k+1}) - \min f) + \frac{1}{2}\distS{z_{k+1}}^2
\leq s_kt_{k+1}(t_{k+1}-1)(f(x_k) - \min f) + \frac{1}{2}\distS{z_k}^2 \\
- s_k\dotp{e_k}{t_{k+1}^2(x_{k+1}-x^\star)-t_{k+1}(t_{k+1}-1)(x_k-x^\star)} .
\end{multline}
We have
\[
t_{k+1}^2(x_{k+1}-x^\star)-t_{k+1}(t_{k+1}-1)(x_k-x^\star) = t_{k+1}(z_{k+1}-x^\star) .
\]
In turn, using also that $s_k$ is non-increasing, \eqref{eq.combinaison_eq_prec_6} becomes
\begin{multline*}
s_{k+1} t_{k+1}^2(f(x_{k+1}) - \min f) + \frac{1}{2}\distS{z_{k+1}}^2 + s_k(t_k^2 - t_{k+1}^2 + t_{k+1})(f(x_k)-\min f)
\leq \\ s_k t_k^2(f(x_k) - \min f) + \frac{1}{2}\distS{z_k}^2
- s_k t_{k+1}\dotp{e_k}{z_{k+1} - x^\star} .
\end{multline*}
In view of the definition of $V_k$, this is equivalent to
\begin{equation}\label{eq.combinaison_eq_prec_7}
V_{k+1} \leq V_k + s_k(t_{k+1}^2-t_{k+1}-t_k^2)(f(x_k)-\min f) + s_k t_{k+1}\dotp{e_k}{z_{k+1} - x^\star} .
\end{equation}
Taking the expectation conditionally on $\filt_k$ in \eqref{eq.combinaison_eq_prec_7}, we obtain
\begin{equation}\label{eq.combinaison_eq_prec_8}
\EX{V_{k+1}}{\filt_k} \leq V_k + s_k(t_{k+1}^2-t_{k+1}-t_k^2)(f(x_k)-\min f) - s_k t_{k+1}\EX{\dotp{e_k}{z_{k+1} - x^\star}}{\filt_k} .
\end{equation}
We have
\begin{align*}
\EX{\dotp{e_k}{z_{k+1} - x^\star}}{\filt_k} 
&= \EX{\dotp{e_k}{z_{k+1} - z_k}}{\filt_k} + \EX{\dotp{e_k}{z_k - z^\star}}{\filt_k} \\
&= -s_k t_{k+1}\EX{\dotp{e_k}{\nabla f_k(y_k)}}{\filt_k} = -s_k t_{k+1}\EX{\dotp{e_k}{\nabla f(y_k)+e_k}}{\filt_k} \\
&= -s_k t_{k+1}\EX{\norm{e_k}^2}{\filt_k}=-s_kt_{k+1}\sigma_k^2,
\end{align*}
where we used \eqref{eq:zkrec} in the second equality, and conditional unbiasedness (first part of \eqref{eq:K2}) in both the second and last inequalities, together with the fact that $y_k$, $z_k$ and $z^\star$ are deterministic conditionally on $\filt_k$. Plugging this into \eqref{eq.combinaison_eq_prec_8} yields
\begin{align}
\EX{V_{k+1}}{\filt_k} 
&\leq V_k + s_k(t_{k+1}^2-t_{k+1}-t_k^2)(f(x_k)-\min f) + s_k^2 t_{k+1}^2\sigma_k^2 \nonumber \\
&\leq V_k + s_k(t_{k+1}^2-t_{k+1}-t_k^2)(f(x_k)-\min f) + 4 s_k^2 t_k^2\sigma_k^2 , \label{eq.combinaison_eq_prec_9}
\end{align}
where we used that assumption \eqref{eq:K1} implies $t_{k+1} \leq 2t_k$; see \cite[Remark~3.3]{AC2R-JOTA}. 
Using again \eqref{eq:K1}, the second term in the rhs of \eqref{eq.combinaison_eq_prec_9} is non-positive and can then be dropped. Now, thanks to the second part of \eqref{eq:K2}, we are in position to apply Lemma~\ref{lem:RS} to \eqref{eq.combinaison_eq_prec_9} to see that $V_k$ converges $\Pas$, and consequently it is bounded $\Pas$. Thus, there exists a $[0,+\infty[$-valued random variable $\xi$ such that $\sup_{k \in \N} V_k \leq \xi < +\infty$ $\Pas$. Therefore, for all $k \geq 1$,
\[
s_k^2 t_{k}^2(f(x_k)-\min f) \leq V_k < +\infty \quad \Pas .
\]
Moreover, taking the total expectation in \eqref{eq.combinaison_eq_prec_9} and iterating gives
\begin{multline*}
s_k^2 t_{k}^2\E{f(x_k)-\min f} \leq \E{V_k} \leq \E{V_1} + 4\sum_{i=1}^{k} s_i^2 t_i^2\E{\norm{e_i}^2} \leq \\ s_1t_1^2\E{f(x_0)-\min f} + \frac{1}{2}\E{\distS{x_0}^2} + 4\sum_{i=1}^{+\infty} s_i^2 t_i^2\E{\norm{e_i}^2} < +\infty ,
\end{multline*}
where we used in the last inequality that $x_0=x_1$ by assumption, and that the rhs is finite thanks to Fubini-Tonelli's Theorem together with \eqref{eq:K2}. This proves the first claim in the theorem.

\smallskip

\noindent\eqref{thm:Thm-Nest-stoch-claim2} 
Using \eqref{eq:K1+} in \eqref{eq.combinaison_eq_prec_9}, we get
\[
\EX{V_{k+1}}{\filt_k}  \leq V_k - s_k(1-m)t_{k+1}(f(x_k)-\min f) + 4 s_k^2 t_k^2\sigma_k^2 .
\]
We can again invoke Lemma~\ref{lem:RS} to get that 
\begin{equation}\label{eq:sumfxk}
\sum_{k \geq 1} s_k t_{k+1}(f(x_k)-\min f) < +\infty \; \Pas .
\end{equation}
Let
\[
W_k \eqdef s_k(f(x_k) - \min f) + \frac{1}{2} \norm{x_k - x_{k-1}}^2 .
\]
Combining \cite[Proposition~2.1]{AC2R-JOTA} with the fact that $s_k$ is non-increasing, we have that
\[
W_{k+1} \leq W_k - \frac{1-\alpha_k^2}{2}\norm{x_k-x_{k-1}}^2 - s_k \dotp{e_k}{x_{k+1}-x_k} .
\]
Taking the expectation conditionally on $\filt_k$, we obtain
\begin{align}\label{eq.combinaison_eq_prec_10}
\EX{W_{k+1}}{\filt_k} \leq W_k - \frac{1-\alpha_k^2}{2}\norm{x_k-x_{k-1}}^2 - s_k \EX{\dotp{e_k}{x_{k+1}-x_k}}{\filt_k} .
\end{align}
We have
\begin{align*}
\EX{\dotp{e_k}{x_{k+1} - x_k}}{\filt_k} 
= \EX{\dotp{e_k}{x_{k+1} - y_k}}{\filt_k} 
= -s_k \EX{\dotp{e_k}{\nabla f_k(y_k)}}{\filt_k}
= -s_k \EX{\norm{e_k}^2}{\filt_k} ,
\end{align*}
where we used the algorithm update of $x_{k+1}$ in the second inequality, and conditional unbiasedness (first part of \eqref{eq:K2}) in the second and last inequalities together with $x_k$, $y_k$ being conditionally deterministic on $\filt_k$. Inserting this into \eqref{eq.combinaison_eq_prec_10} yields
\begin{align}\label{eq.combinaison_eq_prec_10'}
\EX{W_{k+1}}{\filt_k} \leq W_k - \frac{1-\alpha_k^2}{2}\norm{x_k-x_{k-1}}^2 + s_k^2 \sigma_k^2 .
\end{align}
Multiplying \eqref{eq.combinaison_eq_prec_10'} by $t_{k+1}^2$ and rearranging entails
\begin{align}\label{eq.combinaison_eq_prec_11}
&\EX{t_{k+1}^2 W_{k+1}}{\filt_k} 
\leq t_{k+1}^2 W_k - t_{k+1}^2\frac{1-\alpha_k^2}{2}\norm{x_k-x_{k-1}}^2 + s_k^2 t_{k+1}^2 \sigma_k^2 \nonumber\\
&= t_k^2 W_k + s_k (t_{k+1}^2-t_k^2)(f(x_k) - \min f) + \frac{t_{k+1}^2-t_k^2-t_{k+1}^2(1-\alpha_k^2)}{2}\norm{x_k-x_{k-1}}^2 + s_k^2 t_{k+1}^2 \sigma_k^2 \nonumber \\
&\leq t_k^2 W_k + m s_k t_{k+1}(f(x_k) - \min f) - \frac{t_k}{2}\norm{x_k-x_{k-1}}^2 + 4s_k^2 t_k^2 \sigma_k^2 .
\end{align}
In the equality, we used the expression of $W_k$. In the second inequality we used \eqref{eq:K1+} and that $t_k = 1+ t_{k+1}\alpha_k$ and \eqref{eq:K1} which gives
\[
t_{k+1}^2-t_k^2-t_{k+1}^2(1-\alpha_k^2) = (t_k-1)^2 - t_k^2 = -2t_k + 1 \leq - t_k
\]
as $t_k \geq 1$. We have already proved above (see \eqref{eq:sumfxk}) that $\seq{s_k t_{k+1}(f(x_k) - \min f)} \in \ell^1_+(\Filt)$. Combining this with the second part of \eqref{eq:K2} allows us to invoke again Lemma~\ref{lem:RS} on \eqref{eq.combinaison_eq_prec_11} to deduce that 
\begin{equation}\label{eq:sumxkxkm}
\sum_{k \geq 1} t_k \norm{x_k-x_{k-1}}^2 < +\infty \; \Pas .
\end{equation}
Moreover, Lemma~\ref{lem:RS} also implies that $t_k^2 W_k$ converges $\Pas$. On the other hand, we have
\[
t_{k+1}W_k = s_kt_{k+1}(f(x_k) - \min f) + \frac{t_{k+1}}{2} \norm{x_k - x_{k-1}}^2 \leq s_kt_{k+1}(f(x_k) - \min f) + t_k \norm{x_k - x_{k-1}}^2 ,
\]
and thus \eqref{eq:sumfxk} and \eqref{eq:sumxkxkm} imply that 
\[
\sum_{k \geq 1} t_{k+1}W_k < +\infty \; \Pas .
\]
In turn
\[
\sum_{k \geq 1} t_{k+1}W_k = \sum_{k \geq 1} \frac{t_{k+1}}{t_k^2} t_k^2 W_k < +\infty \quad \Pas
\]
entailing that $\liminf_{k \to +\infty} t_k^2 W_k = 0$ $\Pas$. This together with $\Pas$ convergence of $t_k^2 W_k$ shown just above gives that 
\[
W_k = o\pa{\frac{1}{t_k^2}} .
\]
Returning to the definition of $W_k$ proves the assertions.

\smallskip

\noindent\eqref{thm:Thm-Nest-stoch-claim3}
The crux of the proof consists in applying Opial's Lemma on a set of events of probability one. Observe that \eqref{eq:K2+} implies \eqref{eq:K2}. Thus Lemma~\ref{lem:RS} applied to \eqref{eq.combinaison_eq_prec_11} ensures also that $t_k^2 W_k$ converges $\Pas$. In particular, this implies that $t_k\norm{x_k-x_{k-1}}$ is bounded $\Pas$. From the proof of claim (i), we also know that $\Pas$, $V_k$ converges, hence $\seq{z_k}$ is bounded. In view of the definition of $z_k$, we obtain that $\seq{x_k}$ is bounded $\Pas$. Moreover, since $t_k \geq 1$ and $\underline{s}=\inf_k s_k > 0$, we get from (ii) that $\Pas$
\[
\underline{s}\sum_{k \geq 1} (f(x_{k}) - \min f) \leq \sum_{k \geq 1} s_k t_{k+1}(f(x_{k}) - \min f) < +\infty ,
\]
and thus $\lim_{k \to +\infty} f(x_{k}) = \min f$ $\Pas$.

Let $\hat{\Omega}$ be the set of events on which the last statement holds and $\check{\Omega}$ on which boundedness of $\seq{x_k}$ holds. Both sets are of probability one. For any $\omega \in \hat{\Omega} \cap \check{\Omega}$, let $(x_{k_j}(\omega))_{j \geq 1}$ be any converging subsequence, and $\bar{x}(\omega)$ its weak cluster point. 
\[
f(\bar{x}(\omega)) = \lim_{j \to \infty} f(x_{k_j}(\omega)) = \lim_{k \to \infty} f(x_{k}(\omega)) = \min f ,
\]
which means that $\bar{x}(\omega) \in S$. This implies that $\Pas$ each weak cluster point of $\seq{x_k}$ belongs to $S = \argmin(f)$. In other words, the second condition of Opial's lemma holds $\Pas$.

Let $x^\star \in S$ and define $h_k \eqdef \frac{1}{2}\norm{x_k - x^\star}^2$. We now show that $\lim_{k \to +\infty} h_k$ exists $\Pas$. For this, we use a standard argument that can be found e.g. in \cite{AP,AC2R-JOTA}. By \cite[Proposition~2.3]{AC2R-JOTA}, we have
\begin{align*}
h_{k+1} - h_k - \alpha_k(h_k - h_{k-1}) 
&\leq \frac{\alpha_k(1+\alpha_k)}{2}\norm{x_k-x_{k-1}}^2 - s_k (f_k(x_{k+1}) - f_k(x^\star) \nonumber \\
&\leq \norm{x_k-x_{k-1}}^2 - s_k (f(x_{k+1}) - \min f) - s_k \dotp{e_k}{x_{k+1}-x^\star} \nonumber \\
&\leq \norm{x_k-x_{k-1}}^2 - s_k \dotp{e_k}{x_{k+1}-x^\star} . 
\end{align*}
In the second inequality we used that $\alpha_k \in [0,1]$, and the last one minimality of $x^\star$. Almost sure boundedness of $x_k$ implies that there exists a $[0,+\infty[$-valued random variable $\eta$ such that $\sup_{k \in \N} \norm{x_k-x^\star} \leq \eta < +\infty$ $\Pas$.  
Thus 
\begin{align}\label{eq.combinaison_eq_prec_12'}
h_{k+1} - h_k - \alpha_k(h_k - h_{k-1}) \leq \norm{x_k-x_{k-1}}^2 + \eta s_k \norm{e_k} .
\end{align}
Multiplying \eqref{eq.combinaison_eq_prec_12'} by $t_{k+1}$, taking the positive part and the conditional expectation, we end up having
\begin{align*}
\EX{t_{k+1}(h_{k+1} - h_k)_+}{\filt_k} 
&\leq t_{k+1}\alpha_k(h_k - h_{k-1})_+ + t_{k+1}\norm{x_k-x_{k-1}}^2 + \eta s_k t_{k+1} \EX{\norm{e_k}}{\filt_k} \\
&\leq (t_k - 1)(h_k - h_{k-1})_+ + t_{k+1}\norm{x_k-x_{k-1}}^2 + 2\eta s_k t_k \EX{\norm{e_k}^2}{\filt_k}^{1/2} \\
&= t_k(h_k - h_{k-1})_+ - (h_k - h_{k-1})_+ + t_{k+1}\norm{x_k-x_{k-1}}^2 + 2\eta s_k t_k \sigma_k .
\end{align*}
where we used that $t_k=1+t_{k+1}\alpha_k$, that $t_{k+1} \leq 2t_k$ and Jensen's inequality. As the last two terms in the rhs are summable $\Pas$, we get using Lemma~\ref{lem:RS} that $\seq{(h_k - h_{k-1})_+} \in \ell^1_+(\Filt)$ $\Pas$. In turn, since $h_k$ is non-negative, we get by a classical argument that $\lim_{k \to +\infty} h_k$ exists. 

Note that the set of events of probability on which $\lim_{k \to +\infty} h_k$ exists depends on $x^\star$. To make this uniform on $S$ we use a separability argument. 

Indeed, we have just shown that there exists a set of events $\Omega_{x^\star}$ (that depends on $x^\star$) such that $\mathbb{P}(\Omega_{x^\star})=1$ and for all $\omega \in \Omega_{x^\star}$, $\seq{\norm{x_k(\omega)-x^\star}}$ converges.
We now show that there exists a set of events independent of $x^\star$, whose probability is one and such that the above still holds on this set. Since $\cH$ is separable, there exists a countable set $U \subseteq S$, such that $\mathrm{cl}(U)=S$. Let $\tilde{\Omega}=\bigcap_{u \in U}\Omega_u$. Since $U$ is countable, a union bound shows
\[
\mathbb{P}(\tilde{\Omega})=1-\mathbb{P}\pa{\bigcup_{u \in U}\Omega_u^c}\geq 1-\sum_{u \in U}\mathbb{P}(\Omega_u^c) = 1.
\]
For arbitrary $x^{\star}\in S$, there exists a sequence $(u_j)_{j\in\N} \subset U$ such that $u_j$ converges strongly to $x^{\star}$. Thus for every $j\in\N$ there exists $\tau_j: \Omega_{u_j}\rightarrow\R_+$ such that
\begin{equation}\label{eq.combinaison_eq_prec_12}
\lim_{k \to +\infty}\norm{x_k(\omega)-u_j}=\tau_j(\omega), \quad \forall\omega\in\Omega_{u_j}.
\end{equation}
Now, let $\omega \in \tilde{\Omega}$. Since $\tilde{\Omega} \subset \Omega_{u_j}$ for any $j \geq 1$, and  using the triangle inequality and \eqref{eq.combinaison_eq_prec_12}, we obtain that 
\[
\tau_j(\omega) - \norm{u_j-x^{\star}} \leq \liminf_{k \rightarrow +\infty}\norm{x_k(\omega)-x^{\star}} \leq \limsup_{k \to +\infty}\norm{x_k(\omega)-x^{\star}} \leq \tau_j(\omega) + \norm{u_j-x^{\star}} .
\]
Passing to $j \to +\infty$, we deduce
\[
\limsup_{j \to +\infty}\tau_j(\omega) \leq \liminf_{k \to +\infty}\norm{x_k(\omega)-x^{\star}} \leq \limsup_{k \to +\infty}\norm{x_k(\omega)-x^{\star}} \leq \liminf_{j \to +\infty}\tau_j(\omega) ,
\]
whence we deduce that $\lim_{j \to +\infty}\tau_j(\omega)$ exists for all $\omega \in \tilde{\Omega}$. In turn, $\Pas$, $\lim_{k \to +\infty}\norm{x_k-x^{\star}}$ exists and is equal to $\lim_{j \to +\infty}\tau_j$ for any $x^\star \in S$. 

We are now in position to apply Opial's Lemma at any $\omega \in \hat{\Omega} \cap \check{\Omega} \cap \tilde{\Omega}$,  since $\mathbb{P}(\hat{\Omega} \cap \check{\Omega} \cap \tilde{\Omega})=1$, to conclude. \qed

\end{proof}

Let us now return to the Ravine algorithm.
A simple adaptation of the proof of Theorem~\ref{Nest_Ravine} applied to \ref{eq:NAGalphakstoch} (just replace $f$ by $f +\left\langle e_k, \cdot \right\rangle$, and follow similar algebraic manipulations) gives that the associated sequence $\seq{y_k}$  defined by
\[
y_k = x_{k} + \alpha_k ( x_{k} - x_{k-1}),
\]
follows the stochastic Ravine accelerated gradient algorithm with $\gamma_k = \alpha_{k+1}$, i.e. for all $k \geq 1$
\begin{equation}\tag{${\rm(SRAG)}_{\alpha_{k+1}}$}\label{eq:RAGgammakstoch}
\begin{cases}
w_k = y_k - s_k (\nabla f(y_k) + e_k) \vspace{2mm} \\
y_{k+1} = w_k + \alpha_{k+1} \left(w_k - w_{k-1}\right) . \rule{0pt}{5pt} \hspace{3cm}
\end{cases}
\end{equation}
\ref{eq:RAGgammakstoch} is initialized with $y_0$ and $w_{-1}=y_0$, where $y_0$ is a $\cH$-valued, squared integrable random variable.
According to this relationship between the Nesterov and the Ravine method highlighted in in Theorem~\ref{Nest_Ravine}, the results of Theorem~\ref{thm:SNAG} can now be transposed to \ref{eq:RAGgammakstoch}. For this, we denote the canonical filtration associated to \ref{eq:RAGgammakstoch} as $\Filt = \seq{\filt_k}$ with, $\forall k \geq \N$, $\filt_k = \sigma(y_0,(w_{i})_{i \leq k-1})$.

\begin{theorem}\label{thm:SRAG} 
Assume the conditions presented in \eqref{assum:H}. Let $\seq{y_k}$ be the sequence generated by \ref{eq:RAGgammakstoch} where $s_k \in ]0,1/L]$ is a non-increasing sequence, $\seq{\alpha_k} \subset [0,1]$ satisfies \eqref{eq:K0} and \eqref{eq:K1+} with $\sum_{k \in \N} \frac{t_{k+1}}{t_k^2} = +\infty$, and $\seq{e_k}$ is a sequence of stochastic errors satisfying \eqref{eq:K2+}. 
%
%
Then, the sequence $\seq{y_k}$ satisfies
\begin{equation*}
\sum_{k \in \N} s_k t_{k+1}(f(y_k)-\min f) < +\infty \quad \text{and} \quad f(y_k)-  \min_{\cH} f = o \left(\frac{1}{s_k t_{k}^2}\right) \quad \mbox{ as } k\to +\infty \quad \Pas .
\end{equation*}
%
Moreover, if $\inf_k s_k > 0$, then the sequence $\seq{y_k}$ converges weakly $\Pas$ to an $\argmin(f)$-valued random variable.
\end{theorem}

\begin{proof}
According to Theorem~\ref{Nest_Ravine}, the sequence $\seq{x_k}$ defined by 
\begin{equation}\label{def:xk-c}
x_{k+1} =  y_k - s_k (\nabla f (y_k)+e_k)
\end{equation} 
is equivalent to Algorithm \ref{eq:NAGalphakstoch}. It then follows from Theorem~\ref{thm:SNAG}\eqref{thm:Thm-Nest-stoch-claim2} that
\begin{equation}\label{eq:thm:SNAGrecall}
f(x_k)-\min f =o\left(\frac{1}{s_k t_k^2}\right) \quad \mbox{and}\quad \|x_k-\xkm\|=o\left(\frac{1}{t_k}\right) \quad
\Pas . 
\end{equation}
In addition, in view of condition \eqref{eq:K2+}, we can apply Lemma~\ref{lem:RSsum} with $\varepsilon_k=\seq{s_k t_k \sigma_k}$ to infer that
\begin{equation}\label{eq:sumek}
\sum_{k=1}^{+\infty} s_k t_k \norm{e_k} < +\infty \quad \Pas ,
\end{equation}
and thus
\begin{equation}\label{eq:smalloek}
s_k\norm{e_k} = o\pa{\frac{1}{t_k}} \quad \Pas .
\end{equation}
Rearrange the terms in  \eqref{def:xk-c} to obtain the expression $\nabla f (y_k) = -\frac{1}{s_k} (x_{k+1} -  y_k ) - e_k$. Using, successively, the convexity of $f$, the Cauchy-Schwartz inequality, and the triangle inequality, we obtain
\begin{eqnarray}
f(y_k) -\min_{\mathcal H} f &\leq& f(x_k) -\min_{\mathcal H} f + \frac{1}{s_k} \left\langle x_{k+1} -  y_k +s_k e_k, x_k - y_k   \right\rangle
\nonumber \\
&\leq& f(x_k)-\min_{\mathcal H} f  + \frac{1}{s_k} \left( \|x_{k+1} -  y_k\| + s_k \|e_k\|\right)  \|x_{k} -  y_{k}\| \nonumber \\
&\leq& f(x_k)-\min_{\mathcal H} f  + \frac{1}{s_k} \left( \|x_{k+1} -  x_k\| + \|x_{k} -  y_{k}\|  + s_k \|e_k\|\right)  \|x_{k} -  y_{k}\|. \label{NAG2-c}
\end{eqnarray}
Using again the link between \ref{eq:RAGgammakstoch} and \ref{eq:NAGalphakstoch}, we have 
\begin{eqnarray*}
y_{k} &= & x_{k} + \alpha_k \left(x_{k} - x_{k-1}\right).
\end{eqnarray*}
Therefore, since $\alpha_k \in [0,1]$,
\begin{eqnarray}\label{NAG22-c}
\| y_{k}- x_{k} \| \leq \|x_{k} - x_{k-1}\|.
\end{eqnarray}
Combining \eqref{eq:thm:SNAGrecall}, \eqref{eq:smalloek}, \eqref{NAG2-c} and \eqref{NAG22-c} we obtain
\begin{align*}
f(y_k) -\min_{\mathcal H} f 
&\leq f(x_k)-\min_{\mathcal H} f   
+ \frac{1}{s_k} \left( \|x_{k+1} -  x_k\| + \|x_{k} - x_{k-1}\| + s_k \|e_k\|\right)  \|x_{k} - x_{k-1}\| \\
&= o \left(\frac{1}{s_k t_k^2} \right) \quad \Pas 
\end{align*}
where we used that $t_{k+1} \leq 2t_{k}$ in the last equality.
In addition, using Young's inequality, that $\seq{x_k}$ is bounded $\Pas$, \eqref{eq:sumek} and the summability claims of Theorem~\ref{thm:SNAG}\eqref{thm:Thm-Nest-stoch-claim2}, we get that $\Pas$,
\begin{multline*}
\sum_{k \in \N} s_k t_{k+1}(f(y_k)-\min f)
\leq \sum_{k \in \N} s_k t_{k+1}(f(x_k)-\min f)  
+ \sum_{k \in \N} \frac{t_{k+1}}{2}\|x_{k+1} -  x_k\|^2 \\
+ 3\sum_{k \in \N} t_{k}\|x_{k} - x_{k-1}\|^2 + 4\eta\sum_{k \in \N} t_{k}s_k \|e_k\| < +\infty ,  
\end{multline*}
where $\eta$ is the $[0,+\infty[$-valued random variable such that $\sup_{k \in \N} \norm{x_k} \leq \eta < +\infty$ $\Pas$.

Now, from \eqref{eq:thm:SNAGrecall} and \eqref{NAG22-c}, we also have $\norm{y_k - x_k} = o \left(\frac{1}{t_k}\right)$ $\Pas$. Consequently, $y_k-x_k $ converges strongly $\Pas$ to zero. Since the sequence $\seq{x_k}$ converges weakly, it follows that the sequence $\seq{y_k}$ converges weakly $\Pas$ to the same limit as $\seq{x_k}$, and we know from Theorem~\ref{thm:SNAG}\eqref{thm:Thm-Nest-stoch-claim3} that the latter indeed converges weakly $\Pas$ to an $\argmin(f)$-valued random variable. \qed
\end{proof}

\subsection{Fast convergence of the gradients towards zero}
In this section, the previous results on the stochastic Ravine method \ref{eq:RAGgammakstoch} are completed in also showing the fast convergence towards zero of the gradients. 
This will necessitate a specific and intricate Lyapunov analysis\footnote{Observe that embarking from \eqref{eq.combinaison_eq_prec_1a}-\eqref{eq.combinaison_eq_prec_1b} and using the refined estimate in \eqref{eq:extdesclem} is not sufficient to get the result.}. 

Recall $f_k(x) \eqdef f(x) + \dotp{e_k}{x}$ from the proof of Theorem~\ref{thm:SNAG}. The formula in Lemma~\ref{lem:ravseq} hereafter will play a key role in our Lyapunov analysis, and will serve as the constitutive formulation of the algorithm. It corresponds to the Hamiltonian formulation of the algorithm involving the discrete velocities which are defined by, for each $k\in \N$
\begin{equation}\label{def:vk}
v_k \eqdef \frac{1}{h}(y_{k} - y_{k-1})
\end{equation}
where we recall that $h=\sqrt{s}$.

\begin{lemma}\label{lem:ravseq}
Let $\seq{y_k}$ be generated by \ref{eq:RAGgammakstoch}. Then, for all $k\in \N$
\begin{equation}\label{dyn-Rav-2-12-2021_p}
t_{k+1}(v_{k} + h \nabla f_{k-1} (y_{k-1}))- (t_{k} -1) (v_{k-1} + h \nabla f_{k-2} (y_{k-2}))
= - h(t_{k} -1) \nabla f_{k-1} (y_{k-1}).
\end{equation}
\end{lemma}
\begin{proof}
According to the algorithm recursion, we have
\begin{eqnarray*}
y_k &=& y_{k-1} - h^2 \nabla f_{k-1} (y_{k-1}) + \alpha_k \pa{ y_{k-1}- h^2\nabla f_{k-1} (y_{k-1})  - \bpa{y_{k-2}- h^2\nabla f_{k-2} (y_{k-2})}}\\
&= & y_{k-1}  + \alpha_k (y_{k-1}- y_{k-2}) - h^2 \pa{\nabla f_{k-1} (y_{k-1}) + \alpha_k \bpa{ \nabla f_{k-1} (y_{k-1}) -\nabla f_{k-2} (y_{k-2} )}} .
\end{eqnarray*}
Equivalently, 
\begin{align*}
0 &= (y_k - y_{k-1}) -\alpha_k (y_{k-1}- y_{k-2}) + h^2\nabla f_{k-1} (y_{k-1})
+ h^2 \alpha_k ( \nabla f_{k-1} (y_{k-1}) -\nabla f_{k-2} (y_{k-2} )) \\
&= \alpha_k( y_k - y_{k-1}) - \alpha_k(y_{k-1}- y_{k-2}) + (1-\alpha_k) (y_k - y_{k-1}) 
+ h^2 \nabla f_{k-1} (y_{k-1}) \\
&\quad + h^2 \alpha_k ( \nabla f_{k-1} (y_{k-1}) - \nabla f_{k-2} (y_{k-2} )) .
\end{align*}
Let us make $v_k$ appear by multiplying this equality by $\frac{1}{h\alpha_k}$. We then get
\begin{align*}
0 &= v_{k} - v_{k-1}  
 + \frac{1-\alpha_k}{\alpha_k} v_{k} 
+\frac{h}{\alpha_k}\nabla f_{k-1} (y_{k-1})
+ h ( \nabla f_{k-1} (y_{k-1}) -\nabla f_{k-2} (y_{k-2} )) \\
&= (v_{k} + h \nabla f_{k-1} (y_{k-1}))- (v_{k-1} + h \nabla f_{k-2} (y_{k-2}))
+ \frac{1-\alpha_k}{\alpha_k} v_{k} 
+\frac{h}{\alpha_k}\nabla f_{k-1} (y_{k-1}) .
\end{align*}
After multiplication by $\frac{\alpha_k}{1-\alpha_k}$, we arrive at
\begin{multline*}
0 = \frac{\alpha_k}{1-\alpha_k}(v_{k} + h \nabla f_{k-1} (y_{k-1}))- \frac{\alpha_k}{1-\alpha_k}(v_{k-1} + h \nabla f_{k-2} (y_{k-2})) + v_{k} +\frac{h}{1-\alpha_k}\nabla f_{k-1} (y_{k-1}) \\
= \left( 1+\frac{\alpha_k}{1-\alpha_k}\right) (v_{k} + h \nabla f_{k-1} (y_{k-1}))- \frac{\alpha_k}{1-\alpha_k}(v_{k-1} + h \nabla f_{k-2} (y_{k-2}))
 -h \nabla f_{k-1} (y_{k-1}) \\
+\frac{h}{1-\alpha_k}\nabla f_{k-1} (y_{k-1}) .
\end{multline*}
We thus obtain 
\begin{equation*}
\frac{1}{1-\alpha_k}(v_{k} + h \nabla f_{k-1} (y_{k-1}))- \frac{\alpha_k}{1-\alpha_k}(v_{k-1} + h \nabla f_{k-2} (y_{k-2}))
= - \frac{h\alpha_k}{1-\alpha_k}\nabla f_{k-1} (y_{k-1}).
\end{equation*}
Equivalently
\begin{equation}\label{dyn-Rav-1-12-2021_p}
(v_{k} + h \nabla f_{k-1} (y_{k-1}))- \alpha_k(v_{k-1} + h \nabla f_{k-2} (y_{k-2}))
= - h\alpha_k \nabla f_{k-1} (y_{k-1}).
\end{equation}
In view of \eqref{inverse_tk}, the last equality is also equivalent to \eqref{dyn-Rav-2-12-2021_p}. This completes the proof of the Lemma. \qed
\end{proof}

Recall the canonical filtration associated to \ref{eq:RAGgammakstoch} as $\Filt = \seq{\filt_k}$ with, $\forall k \geq \N$, $\filt_k = \sigma(y_0,(w_{i})_{i \leq k-1})$.
\begin{theorem}\label{thm:SRAGgrad}
Let us assume the conditions defined in \eqref{assum:H}. Let $\seq{y_k}$ be the sequence generated by \ref{eq:RAGgammakstoch} where $s_k \equiv s \in ]0,1/L]$, $\seq{\alpha_k} \subset [0,1]$ satisfy \eqref{eq:K0} and \eqref{eq:K1+}. Assume that $\seq{e_k}$ is a sequence of stochastic errors subject to conditions \eqref{eq:K2+}. 
%
%
Then the sequence of gradients $\seq{\nabla f (y_k)}$ converges to zero with
\[
\sum_{k \in \N}  t_{k+1}^2\| \nabla f (y_k) \|^2 < +\infty \quad \Pas .
\]
%
%
\end{theorem}


\begin{proof}
Our Lyapunov analysis is based on the sequence $\seq{E_k}$ defined as
\begin{eqnarray*}
E_k &\eqdef&   h^2 (t_{k+1}-1)t_{k+1}( f(y_{k-1})-  \min f ) +\frac{1}{2}\distS{z_k}^2 , \\
z_k &\eqdef&  y_{k} + h(t_{k+1}-1)\Big( v_{k} + h \nabla f_{k-1}( y_{k-1})  \Big).
\end{eqnarray*}
Let $x^\star$ be the closest point to $z_k$ in $S$. By definition of $ E_k$, we have
\begin{multline}
E_{k+1} - E_k 
\leq h^2 (t_{k+1}-1)t_{k+1}( f (y_{k})-  f (y_{k-1}) ) \\
+h^2 \Big((t_{k+2}-1)t_{k+2} - (t_{k+1}-1)t_{k+1}\Big) ( f (y_{k})-  \min f )
+\frac{1}{2}\|z_{k+1}-x^\star\|^2 -\frac{1}{2}\|z_{k}-x^\star\|^2   \label{Lyap-10} . 
\end{multline}
Let us compute this last expression with the help of the elementary inequality
\begin{equation}\label{elem-ineq}
\frac{1}{2}\|z_{k+1}-x^\star\|^2 -\frac{1}{2}\|z_k-x^\star\|^2  = \left\langle  z_{k+1} - z_{k} , z_{k+1} - x^\star \right\rangle - \frac{1}{2}\|z_{k+1} - z_{k}\|^2  .
\end{equation}
Recall  the constitutive equation given by \eqref{dyn-Rav-2-12-2021_p} that we write as follows
\begin{equation}\label{dyn-Rav-2-12-2021-b}
t_{k+2}(v_{k+1} + h \nabla f_{k} (y_{k}))- (t_{k+1} -1) (v_{k} + h \nabla f_{k-1} (y_{k-1}))
= - h(t_{k+1} -1) \nabla f_{k} (y_{k}).
\end{equation}
Using successively the definition of $z_k$ and  \eqref{dyn-Rav-2-12-2021-b}, we obtain
\begin{eqnarray*}
z_{k+1} - z_{k}&=& (y_{k+1}  - y_{k})+ h(t_{k+2}-1)\Big( v_{k+1} + h \nabla f_{k}( y_{k})  \Big) - h(t_{k+1}-1)\Big( v_{k} + h \nabla f_{k-1}( y_{k-1})  \Big)  \\
&=& h v_{k+1} - h\Big( v_{k+1} + h \nabla f_{k}( y_{k})  \Big) \
-h^2 (t_{k+1} -1) \nabla f_{k} (y_{k}) 
= - h^2 t_{k+1}  \nabla f_{k} (y_{k}). 
\end{eqnarray*}
This together with the definition of $z_k$ yields
\[
z_{k+1}=  z_k - h^2 t_{k+1}  \nabla f_{k} (y_{k}) = y_{k} + h(t_{k+1}-1)\Big( v_{k} + h \nabla f_{k-1}( y_{k-1})  \Big)
- h^2 t_{k+1} \nabla f_{k} (y_{k}) .
\]
Plugging this into \eqref{elem-ineq}, we deduce that
\begin{eqnarray*}
&&\frac{1}{2}\|z_{k+1}-x^\star\|^2 -\frac{1}{2}\|z_k-x^\star\|^2  =  - \frac{1}{2}h^4 {t}^2_{k+1} \| \nabla f_{k} (y_{k}) \|^2 \\
&& - h^2t_{k+1} \left\langle   \nabla f_{k} (y_{k}) , y_{k}  - x^\star + h(t_{k+1}-1)\Big( v_{k} + h \nabla f_{k-1}( y_{k-1})  \Big)
- h^2 t_{k+1} \nabla f_{k} (y_{k}) \right\rangle\\
&&= \demi  h^4 {t}^2_{k+1} \| \nabla f_{k} (y_{k}) \|^2 
 -  h^2 t_{k+1} \left\langle   \nabla f_{k} (y_{k}) ,  y_{k}  - x^\star + h(t_{k+1}-1)\Big( v_{k} + h \nabla f_{k-1}( y_{k-1})  \Big)\right\rangle.
\end{eqnarray*}
Let us arrange the above expression so as to group the products of  $\nabla f_{k} (y_{k})$. For this, we use \eqref{dyn-Rav-2-12-2021_p} again, written as,
\begin{equation}\label{dyn-Rav-2-12-2021-d}
 (t_{k+1} -1) (v_{k} + h \nabla f_{k-1} (y_{k-1}))
= t_{k+2}(v_{k+1} + h \nabla f_{k} (y_{k})) + h(t_{k+1} -1) \nabla f_{k} (y_{k}).
\end{equation}
Therefore,
\begin{align*}
&y_{k}  - x^\star + h(t_{k+1}-1)\Big( v_{k} + h \nabla f_{k-1}( y_{k-1})  \Big) \\
&= y_{k}  - x^\star +ht_{k+2}(v_{k+1} + h \nabla f_{k} (y_{k})) + h^2(t_{k+1} -1) \nabla f_{k}(y_{k}) \\
&= y_{k}  - x^\star + ht_{k+2} v_{k+1} + h^2 (t_{k+2} + t_{k+1} -1) \nabla f_{k} (y_{k}).
\end{align*}
Collecting the above results we obtain
\begin{multline*}
\frac{1}{2}\|z_{k+1}-x^\star\|^2 -\frac{1}{2}\|z_k-x^\star\|^2 = \demi h^4 {t}^2_{k+1} \| \nabla f_{k} (y_{k}) \|^2 \\
 -  h^2 t_{k+1} \left\langle   \nabla f_{k} (y_{k}) ,  y_{k}  - x^\star + ht_{k+2} v_{k+1} 
 +h^2 (t_{k+2} + t_{k+1} -1) \nabla f_{k} (y_{k})\right\rangle.
\end{multline*}
Inserting this in \eqref{Lyap-10} we get
\begin{align}
&E_{k+1} - E_k 
\leq h^2 (t_{k+1}-1)t_{k+1}( f (y_{k})-  f (y_{k-1}) )  \nonumber \\
& + h^2 \Big((t_{k+2}-1)t_{k+2} - (t_{k+1}-1)t_{k+1}\Big) ( f (y_{k})-  \min f ) +\demi  h^4 {t}^2_{k+1} \| \nabla f_{k} (y_{k}) \|^2   \nonumber \\
& - h^2 t_{k+1} \left\langle   \nabla f_{k} (y_{k}) ,  y_{k}  - x^\star + ht_{k+2} v_{k+1} 
  + h^2 (t_{k+2} + t_{k+1} -1) \nabla f_{k} (y_{k})\right\rangle.
 \label{basicEk} 
\end{align}
In view of the basic gradient inequality for convex differentiable functions whose gradient is $L$-Lipschitz continuous, we have
\begin{align*}
f (y_{k-1}) &\geq  f (y_{k})+  \left\langle   \nabla f (y_{k}) ,  y_{k-1} - y_{k} \right\rangle + \frac{1}{2L}\|\nabla f (y_{k}) -\nabla f (y_{k-1})  \|^2. \\
\min f &\geq  f (y_{k})+  \left\langle   \nabla f (y_{k}) ,  x^\star - y_{k} \right\rangle + \frac{1}{2L}\norm{\nabla f(y_{k})}^2 .
\end{align*}
Combining the above inequalities with \eqref{basicEk}, and using $\nabla f_{k} (y_{k}) = \nabla f (y_{k}) +  e_{k}$,
we get
\begin{align}
& E_{k+1} - E_k \leq -h^2 (t_{k+1}-1)t_{k+1}\Big(\left\langle \nabla f (y_{k}) ,  y_{k-1} - y_{k} \right\rangle + \frac{1}{2L}\|\nabla f (y_{k}) -\nabla f (y_{k-1})  \|^2 \Big) \nonumber\\
& +h^2 \Big((t_{k+2}-1)t_{k+2} - (t_{k+1}-1)t_{k+1}\Big) ( f (y_{k})-  \min f ) - h^2 t_{k+1}\pa{f (y_{k}) - \min f}  \nonumber\\
&  +\demi  h^4 {t}^2_{k+1} \| \nabla f_{k} (y_{k}) \|^2 -  h^2 t_{k+1} \left\langle   \nabla f_{k} (y_{k}) ,   ht_{k+2} v_{k+1} +h^2 (t_{k+2} + t_{k+1} -1) \nabla f_{k} (y_{k})\right\rangle \nonumber\\
& -h^2 t_{k+1} \dotp{y_{k}-x^\star}{e_k} . \label{basicEk_2} 
\end{align}
Next rearrange the last inequality by grouping terms on the right hand side with common expressions. To begin with, rewrite the second and third summand as follows:
\begin{multline*}
h^2 \Big((t_{k+2}-1)t_{k+2} - (t_{k+1}-1)t_{k+1}\Big) ( f (y_{k})-  \min f ) - h^2 t_{k+1}(f (y_{k}) - \min f) 
=  \\
-h^2 \Big(t_{k+1}^2 - t_{k+2}^2 +  t_{k+2}\Big)( f (y_{k})-  \min f ) .
\end{multline*}
For the following expression grouping two of the summands above, we use the definition of $v_{k}$ for the first equality, and  the constitutive equation
\eqref{dyn-Rav-2-12-2021-d} for the third,
\begin{align*}
& -h^2 (t_{k+1}-1)t_{k+1}\left\langle   \nabla f (y_{k}) ,  y_{k-1} - y_{k} \right\rangle -  h^2 t_{k+1} \left\langle   \nabla f_{k} (y_{k}) ,   ht_{k+2} v_{k+1} 
\right\rangle 
 \\
&= h^3 (t_{k+1}-1)t_{k+1}\left\langle   \nabla f (y_{k}) ,  v_{k}\right\rangle -  h^3 t_{k+1} t_{k+2}\left\langle   \nabla f_{k} (y_{k}) , v_{k+1} \right\rangle \\
&= h^3 t_{k+1} \left\langle   \nabla f (y_{k}) , (t_{k+1}-1) v_{k}- t_{k+2} v_{k+1} \right\rangle -  h^3 t_{k+1} t_{k+2} \dotp{v_{k+1}}{e_{k}}    \\
&=  h^3t_{k+1} \Big(\left\langle   \nabla f (y_{k}) , - h (t_{k+1} -1) \nabla f_{k-1} (y_{k-1}) +h(t_{k+1} +t_{k+2} -1) \nabla f_{k} (y_{k}) \right\rangle \Big) -  h^3 t_{k+1} t_{k+2} \dotp{v_{k+1}}{e_{k}} \\
&=  h^4t_{k+1} \Big(\left\langle   \nabla f (y_{k}) , - (t_{k+1} -1) \nabla f (y_{k-1}) + (t_{k+1} +t_{k+2} -1) \nabla f (y_{k}) \right\rangle \Big) \\
&-  h^3 t_{k+1} t_{k+2} \dotp{v_{k+1}}{e_{k}} 
+ h^4t_{k+1}(t_{k+1} +t_{k+2} -1) \dotp{\nabla f (y_{k})}{e_k}
- h^4 t_{k+1}(t_{k+1} -1) \dotp{\nabla f (y_{k})}{e_{k-1}} .
\end{align*}

In addition
\begin{equation*}
\demi  h^4 {t}^2_{k+1} \| \nabla f_{k} (y_{k}) \|^2 
-h^4 {t}_{k+1}(t_{k+2} + t_{k+1} -1)\| \nabla f_{k} (y_{k}) \|^2  
= -\frac{1}{2}h^4 t_{k+1} (2t_{k+2} + t_{k+1} -2)\| \nabla f_{k} (y_{k}) \|^2.
\end{equation*}
Collecting the last three estimates and applying the inequalities to \eqref{basicEk_2}, we obtain
\begin{align*}
& E_{k+1} - E_k  +h^2 \Big(t_{k+1}^2 - t_{k+2}^2 +  t_{k+2}\Big)( f (y_{k})-  \min f )\\
&\leq -\frac{h^2}{2L}(t_{k+1}-1)t_{k+1}  \|\nabla f (y_{k}) -\nabla f (y_{k-1})  \|^2  \\
&+ h^4t_{k+1} \Big(\left\langle   \nabla f (y_{k}) ,
  - (t_{k+1} - 1) \nabla f (y_{k-1}) + (t_{k+1} +t_{k+2} -1) \nabla f (y_{k}) \right\rangle \Big)\\
& -\frac{1}{2}h^4 t_{k+1} (2t_{k+2} + t_{k+1} -2)\| \nabla f (y_{k}) +  e_{k}\|^2\\
& -h^2 t_{k+1} \left\langle e_{k},y_{k}-x^\star\right\rangle
-  h^3 t_{k+1} t_{k+2} \dotp{v_{k+1}}{e_{k}} \\
&+ h^4t_{k+1}(t_{k+1} +t_{k+2} -1) \dotp{\nabla f (y_{k})}{e_k}
- h^4t_{k+1}(t_{k+1} -1) \dotp{\nabla f (y_{k})}{e_{k-1}} .
\end{align*}
After developing the expression $\|\nabla f (y_{k}) +  e_k\|^2$, we arrive at
\begin{align*}
& E_{k+1} - E_k  +h^2 \Big(t_{k+1}^2 - t_{k+2}^2 +  t_{k+2}\Big)( f (y_{k})-  \min f )\\
&\leq -\frac{h^2}{2L}(t_{k+1}-1)t_{k+1}  \|\nabla f (y_{k}) -\nabla f (y_{k-1})  \|^2  \\
&+ h^4t_{k+1} \Big(\left\langle   \nabla f (y_{k}) ,
  - (t_{k+1}-1) \nabla f (y_{k-1}) + (t_{k+1} +t_{k+2} -1) \nabla f (y_{k}) \right\rangle \Big)\\
& -\frac{1}{2}h^4 t_{k+1} (2t_{k+2} + t_{k+1} -2)\Big( \| \nabla f (y_{k})\|^2  + \| e_{k}\|^2 + 2\dotp{\nabla f (y_{k})}{e_k} \Big) \\
& -h^2 t_{k+1} \left\langle e_{k},y_{k}-x^\star\right\rangle
-  h^3 t_{k+1} t_{k+2} \dotp{v_{k+1}}{e_{k}} \\
&+ h^4t_{k+1}(t_{k+1} +t_{k+2} -1) \dotp{\nabla f (y_{k})}{e_k}
- h^4t_{k+1}(t_{k+1} -1) \dotp{\nabla f (y_{k})}{e_{k-1}} .
\end{align*}
Taking the expectation conditionally on $\filt_k$ and using conditional unbiasedness in \eqref{eq:K2+}, we get that $\Pas$
\begin{align*}
& \EX{E_{k+1}}{\filt_k} - E_k  + h^2 \Big(t_{k+1}^2 - t_{k+2}^2 +  t_{k+2}\Big)( f(y_{k})-  \min f )\\
&\leq -\frac{h^2}{2L}(t_{k+1}-1)t_{k+1}  \|\nabla f (y_{k}) -\nabla f (y_{k-1})  \|^2  \\
&+ h^4t_{k+1} \Big(\left\langle   \nabla f (y_{k}) ,
  - (t_{k+1} -1) \nabla f (y_{k-1}) + (t_{k+1} +t_{k+2} -1) \nabla f (y_{k}) \right\rangle \Big)\\
& -\frac{1}{2}h^4 t_{k+1} (2t_{k+2} + t_{k+1} -2)\| \nabla f (y_{k})\|^2  -\frac{1}{2}h^4 t_{k+1} (2t_{k+2} + t_{k+1} -2) \sigma_k^2 \\
&+  h^3 t_{k+1} t_{k+2} \EX{\norm{v_{k+1}}^2}{\filt_k}^{1/2} \sigma_k ,
\end{align*}
where we used Cauchy-Schwartz inequality in the last term. Now we rely on Theorem~\ref{thm:SNAG}, and in particular on \eqref{eq:thm:SNAGrecall} and \eqref{NAG22-c} to infer that
\begin{align*}
\|v_{k+1} \| 
= \frac{1}{h}\| y_{k+1} - y_{k}\|
&\leq \frac{1}{h}\|y_{k+1} - x_{k+1}\| + \frac{1}{h}\|x_{k+1} - x_{k}\| +  \frac{1}{h}\|x_{k} - y_{k}\| \\
&\leq \frac{2}{h}\|x_{k+1} - x_{k}\| + \frac{1}{h}\|x_{k} -x_{k-1}\| = o \left( \frac{1}{t_{k+1}} \right) + o \left( \frac{1}{t_{k}} \right) = o \left( \frac{1}{t_{k+1}} \right) \quad \Pas .
\end{align*}
In the last equality we used again that \eqref{eq:K1+} implies $t_{k+1} \leq 2t_{k}$. 
Therefore, there exists a non-negative random variable $\eta$ with $\esssup \eta < +\infty$ such that $\EX{\norm{v_{k+1}}^2}{\filt_k}^{1/2} \leq \eta/t_{k+1}$, and in turn
\begin{align*}
& \EX{E_{k+1}}{\filt_k} - E_k  + h^2 \Big(t_{k+1}^2 - t_{k+2}^2 +  t_{k+2}\Big)( f(y_{k})-  \min f )\\
&\leq -\frac{h^2}{2L}(t_{k+1}-1)t_{k+1}  \|\nabla f (y_{k}) -\nabla f (y_{k-1})  \|^2  \\
&+ h^3t_{k+1} \Big(\left\langle   \nabla f (y_{k}) ,
  - h (t_{k+1} -1) \nabla f (y_{k-1}) +h(t_{k+1} +t_{k+2} -1) \nabla f (y_{k}) \right\rangle \Big)\\
& -\frac{1}{2}h^4 t_{k+1} (2t_{k+2} + t_{k+1} -2)\| \nabla f (y_{k})\|^2  
+  4\eta h^3 t_{k} \sigma_k ,
\end{align*}
where we used again that $t_{k+2} \leq 4t_{k}$ and we discarded the term involving $\sigma_k^2$ since $t_k \geq 1$ and thus $2t_{k+2} + t_{k+1} -2 \geq 1$. Equivalently,
\begin{equation}\label{eq:basic_Lyap_040222}
\EX{E_{k+1}}{\filt_k} - E_k  + h^2 \Big(t_{k+1}^2 - t_{k+2}^2 +  t_{k+2}\Big)( f(y_{k})-  \min f )
\leq -R( \nabla f (y_{k-1}), \nabla f (y_{k})) + 4\eta h^3 t_{k} \sigma_k ,
\end{equation}
where $R$ is the quadratic form 
\begin{multline}
R(X,Y) = \frac{h^2}{2L}(t_{k+1}-1)t_{k+1} \|Y-X\|^2  +\frac{1}{2}h^4 t_{k+1} (2t_{k+2} + t_{k+1} - 2) \|Y\|^2  \\
- h^3t_{k+1}\pa{\dotp{Y}{- h (t_{k+1} -1) X +h(t_{k+1} +t_{k+2} -1) Y}} . \label{eq:basic_Lyap_3112}
\end{multline}
To conclude, we just need to prove that $R$ is nonnegative. A standard procedure consists in computing a lower-bound $\min_{X} R(X,Y)$ for fixed $Y$. By taking the derivative of $R$ with respect to $X$, we obtain that the minimum is achieved at $\bar{X}$ with $\bar{X}-Y= -h^2 LY$. Therefore, 
\begin{align*}
\min_X R(X,Y) 
&= \frac{h^2 L}{2}(t_{k+1}-1)t_{k+1} h^4\|Y\|^2  +\frac{1}{2}h^4 t_{k+1} (2t_{k+2} + t_{k+1} - 2)  \| Y \|^2\\
&- h^3t_{k+1}\pa{\dotp{Y}{- h (t_{k+1} -1)(1- h^2 L)Y +h(t_{k+1} + t_{k+2} - 1) Y}}.  
\end{align*}
After reduction, we get
\begin{equation} \label{R_positive}
\min_X R(X,Y) = \frac{h^4 t_{k+1}}{2}\pa{(t_{k+1} - 1) (2-h^2L) - 1}\|Y\|^2 .
\end{equation}
According to assumption \eqref{eq:K1+}, the coefficient of $f(y_{k}) - \min f$ in \eqref{eq:basic_Lyap_040222} is positive. We therefore discard this term in the rest of the proof. Combining \eqref{R_positive} with \eqref{eq:basic_Lyap_040222}, we obtain
\begin{equation*}
\EX{E_{k+1}}{\filt_k} - E_k 
\leq -\frac{h^4 t_{k+1}}{2}\pa{(t_{k+1} - 1) (2-h^2L) - 1}\|\nabla f(y_{k})\|^2 + 4\eta h^3 t_{k} \sigma_k .
\end{equation*}
Since $h^2 \in ]0,1/L]$ and $t_k \geq 1$, this can also be bounded as 
\begin{align*}
\EX{E_{k+1}}{\filt_k} 
&\leq E_k - \frac{h^2 t_{k+1}}{2L}\pa{(t_{k+1} -1) (2-h^2L) - 1}\|\nabla f(y_{k})\|^2 + \frac{4\eta h}{L} t_{k} \sigma_k\\
&\leq E_k - \frac{h^2 t_{k+1}(t_{k+1}-2)}{2L} \|\nabla f(y_{k})\|^2 + \frac{4\eta h}{L} t_{k} \sigma_k \\
&= E_k - \frac{h^2 t_{k+1}^2}{2L} \|\nabla f(y_{k})\|^2 + \frac{h^2 t_{k+1}}{L} \|\nabla f(y_{k})\|^2 + \frac{4\eta h}{L} t_{k} \sigma_k \\
&\leq E_k - \frac{h^2 t_{k+1}^2}{2L} \|\nabla f(y_{k})\|^2 + 2h^2 t_{k+1}(f(y_{k})-  \min f) + \frac{4\eta h}{L} t_{k} \sigma_k ,
\end{align*}
where we used co-coercivity of $\nabla f$ in the last inequality. The summability assumption in \eqref{eq:K2+} together with the summability result in Theorem~\ref{thm:SRAG} allow then to invoke Lemma~\ref{lem:RS} to get the claim. Observe that this also gives that $E_k$ converges $\Pas$ to a non-negative valued random variable. \qed
\end{proof}

\begin{remark}{
Since $t_k \geq 1$, a direct consequence of the gradient summability shown in Theorem~\ref{thm:SRAGgrad} is that the gradient sequence $\seq{\nabla f(y_k)}$ tends to zero $\Pas$ at least as quickly as at the rate $o(1/t_k)$. Observe also that this analysis gives another proof for the fast convergence of the function values (just carry on the proof starting from \eqref{eq:basic_Lyap_040222} without discarding the term involving the function values). 

Note that the above proof has been notably simplified by using the conclusions already obtained in Theorem~\ref{thm:SRAG}, and in particular to properly bound the terms involving $v_{k+1}$ (which are not in $\filt_k$). Extending this proof to the case where the step-size $s_k$ is varying appears to be straightforward, but comes at the price of tedious and longer computations. We avoid this for the sake of brevity.
}
\end{remark}

\section{Application to Particular Parameter Choices}\label{sec:special_cases}
Let consider the theoretical guarantees obtained under the condition that there exists $c\in [0,1[$ such that, for every $k\geq 1$ 
\begin{equation}\label{eq.special_class}
\frac{1}{1-\alpha_{k+1}}-\frac{1}{1-\alpha_{k}}\leq c .
\end{equation}
This implies some important properties of $t_k$. One significant observation is a trade-off between stability to errors and fast convergence of \eqref{eq:RAGgammakstoch}. Some choices of $\alpha_k$ will be less stringent on the required summability of the error variance for convergence, but will result in slower convergence rate and vice-versa.

In presenting the details, let us start with the following results that were obtained in \cite[Proposition 3.3, 3.4]{AC2}. The first one presents some general conditions on $(\alpha_k)$ and $c$ that ensure the satisfaction of \eqref{eq:K0} and \eqref{eq:K1} (resp. \eqref{eq:K1+}). The second one provides an explicit expression of $t_k$ as a function of $\alpha_k$. 
\begin{proposition}\label{pr.special_class_ineq}
Let $c\in[0,1[$ and let $\seq{\alpha_k}$ be a sequence satisfying $\alpha_k\in [0,1[$ together with inequality {\rm(\ref{eq.special_class})} for every $k\geq 1$. Then condition \eqref{eq:K0} is satisfied. Moreover, we have for every $k\geq 1$,
\[
t_{k+1}\leq \frac{1}{(1-c)(1-\alpha_k)}.
\]
If $c\leq 1/3$ (resp. $c<1/3$), then condition \eqref{eq:K1} (resp. \eqref{eq:K1+}) is fulfilled.
\end{proposition}
\begin{proposition}\label{pr.special_class_equiv}
Let $\seq{\alpha_k}$ be a sequence such that $\alpha_k\in [0,1[$ for every $k\geq 1$. Given $c\in [0,1[$, assume that
\begin{equation*}
\lim_{k\to +\infty}\frac{1}{1-\alpha_{k+1}}-\frac{1}{1-\alpha_{k}}=c.
\end{equation*}
Then, we have
\[
t_{k+1}\sim \frac{1}{(1-c)(1-\alpha_k)} \quad \mbox{ as } k\to +\infty.
\]
\end{proposition}

\noindent
Let us now consider several possible iterative regimes defining $\alpha_k$.
\subsection{Case 1: $\displaystyle{\alpha_k=1-\frac{\alpha}{k}}$, $\alpha>0$:}
This corresponds to the choice made in the (deterministic) Nesterov and Ravine methods studied in \cite{AF}.
In this case, for every $k\geq 1$,
\[
\frac{1}{1-\alpha_{k+1}}-\frac{1}{1-\alpha_{k}}=\frac{k+1}{\alpha}-\frac{k}{\alpha}=\frac{1}{\alpha}.
\]
Therefore, condition \eqref{eq.special_class} is satisfied with $c=\frac{1}{\alpha}$. If $\alpha\geq 3$ (resp. $\alpha>3$), we have $c\in ]0,1/3]$ (resp. $c\in ]0,1/3[$). According to Proposition \ref{pr.special_class_equiv}, we have for every $k\geq 1$,
\[
t_{k+1}\sim \frac{1}{(1-c)(1-\alpha_k)}=\frac{\alpha}{\alpha-1}\,\frac{k}{\alpha}=\frac{k}{\alpha-1}.
\]
Indeed, one can easily show that the equality $t_{k+1}=\frac{k}{\alpha-1}$ is satisfied. Moreover, 
\[
t_{k+1}/t_k^2 = k(\alpha-1)/(k-1)^2 \geq (\alpha-1)/(k-1) \Rightarrow \sum_{k \in \N} \frac{t_{k+1}}{t_k^2} = +\infty .
\]
Thus, specializing Theorem~\ref{thm:SRAG} and Theorem~\ref{thm:SRAGgrad}, we obtain the following statement.
\begin{corollary}\label{co.alpha_k=1-alpha/k}
Assume that \eqref{assum:H} holds. Let $\seq{y_k}$ be the sequence generated by \ref{eq:RAGgammakstoch} with $\alpha_k=1-\frac{\alpha}{k}$ where $\alpha > 3$, and $s_k \in ]0,1/L]$ is a non-increasing sequence. Assume that 
\[
\EX{e_k}{\filt_k} = 0 \; \Pas \quad \text{and} \quad \seq{k s_k \sigma_k} \in \ell^1_+(\Filt) .
\] 
Then, the following holds $\Pas$:
\begin{enumerate}[(i)]
\item 
$f(y_k)-  \min_{\cH} f = o \left(\frac{1}{s_k k^2}\right)$ and $\|y_k-y_{k-1}\|=o\left(\frac{1}{k}\right)$ ;

\item
$\displaystyle\sum_{k \in \N}ks_k(f(y_k)-\min_{\cH} f)<+\infty$ and $\displaystyle\sum_{k \in \N}k\|y_k-y_{k-1}\|^2<+\infty$ ;

\item If moreover $\inf_k s_k > 0$, then $\sum_{k \in \N}  k^2 \|\nabla f (y_k) \|^2 < +\infty$ and $\seq{y_k}$ converges weakly $\Pas$ to an $\argmin(f)$-valued random variable.
\end{enumerate}
\end{corollary}

Another possible choice would be $\alpha_k=\frac{k}{k+\alpha}$ in which case we obtain exactly the same results as in Corollary~\ref{co.alpha_k=1-alpha/k}. This corresponds to the popular choice of the the Nesterov extrapolation parameter. For \ref{eq:NAGalphakstoch} with this choice of $\alpha_k$, we recover and complete the results obtained in the literature; see e.g., \cite{AR,AZ,LanBook,Laborde2020}.

\subsection{Case 2: $\displaystyle{\alpha_k=1-\frac{\alpha}{k^r}}$, $\alpha>0$, $r\in ]0,1[$:}
In this case, we have
\[
\frac{1}{1-\alpha_{k+1}}-\frac{1}{1-\alpha_{k}}=\frac{1}{\alpha}(k+1)^r-\frac{1}{\alpha}k^r=\frac{k^r}{\alpha}\left((1+1/k)^r-1\right)\sim \frac{r}{\alpha}\, k^{r-1}\to 0 \quad \mbox{as } k\to +\infty.
\]
For each $c>0$, the condition $1/(1-\alpha_{k+1})-1/(1-\alpha_k)\leq c$ is satisfied for $k$ large enough. On the other hand, we deduce from Proposition \ref{pr.special_class_equiv} that 
$\displaystyle{t_k\sim \frac{k^r}{\alpha}}$ as $k\to +\infty$. This implies that $\displaystyle{\sum_{i=1}^k t_i\sim \frac{1}{\alpha(1+r)}\,k^{1+r}}$  as $k\to +\infty$.
Theorem~\ref{thm:SRAG} and Theorem~\ref{thm:SRAGgrad} under this specification yields the following result.
\begin{corollary}\label{co.alpha_k=1-alpha/k^r}
Assume that \eqref{assum:H} holds. Let $\seq{y_k}$ be the sequence generated by \ref{eq:RAGgammakstoch} with $\alpha_k=1-\frac{\alpha}{k^r}$ where $\alpha > 0$ and $r \in ]0,1[$, and $s_k \in ]0,1/L]$ is an non-increasing sequence. Assume that 
\[
\EX{e_k}{\filt_k} = 0 \; \Pas \quad \text{and} \quad \seq{k^r s_k \sigma_k} \in \ell^1_+(\Filt) .
\] 
Then, the following holds $\Pas$:
\begin{enumerate}[(i)]
\item 
$f(y_k)-  \min_{\cH} f = o \left(\frac{1}{s_k k^{2r}}\right)$ and $\|y_k-y_{k-1}\|=o\left(\frac{1}{k^r}\right)$ ;

\item
$\displaystyle\sum_{k \in \N}k^rs_k(f(y_k)-\min_{\cH} f)<+\infty$ and $\displaystyle\sum_{k \in \N}k^r\|y_k-y_{k-1}\|^2<+\infty$ ;

\item If moreover $\inf_k s_k > 0$, then $\sum_{k \in \N}  k^{2r} \|\nabla f (y_k) \|^2 < +\infty$ and $\seq{y_k}$ converges weakly $\Pas$ to an $\argmin(f)$-valued random variable.
\end{enumerate}
\end{corollary}

It is clear from this result that this choice of $\alpha_k$ allows for a less stringent summability condition on the stochastic errors, but this comes at the price of a slower convergence rate. We are not aware of any such a result in the literature.

\subsection{Case 3: $\alpha_k$ constant:}
This corresponds to the choice made in the Polyak's heavy ball with friction method \cite{Polyak_1,Polyak_2}. Since $\alpha_k \equiv \alpha\in [0,1[$ for every $k\geq 1$, condition \eqref{eq.special_class} is clearly satisfied with $c=0$. In turn, $t_k \equiv 1/(1-\alpha)$ for all $k \geq 1$. Applying Theorem~\ref{thm:SRAG} and Theorem~\ref{thm:SRAGgrad} we get the following.
\begin{corollary}\label{co.alpha_k_non-increasing}
Assume that \eqref{assum:H} holds. Let $\seq{y_k}$ be the sequence generated by \ref{eq:RAGgammakstoch} with $\alpha_k \equiv \alpha \in [0,1[$, and $s_k \in ]0,1/L]$ is an non-increasing sequence. Assume that 
\[
\EX{e_k}{\filt_k} = 0 \; \Pas \quad \text{and} \quad \seq{s_k \sigma_k} \in \ell^1_+(\Filt) .
\] 
Then, the following holds $\Pas$:
\begin{enumerate}[(i)]
\item
$\displaystyle\sum_{k \in \N}s_k(f(y_k)-\min_{\cH} f)<+\infty$ and $\displaystyle\sum_{k \in \N}\|y_k-y_{k-1}\|^2<+\infty$ ;

\item If moreover $\inf_k s_k > 0$, then $\sum_{k \in \N} \|\nabla f (y_k) \|^2 < +\infty$ and $\seq{y_k}$ converges weakly $\Pas$ to an $\argmin(f)$-valued random variable.
\end{enumerate}
\end{corollary}

For \ref{eq:NAGalphakstoch} with this choice of $\alpha_k$, we recover and complete the results obtained in the literature; see e.g., \cite{Yang2016,Gadat2018,Loizou2020,Sebbouh2021}.



\section{Conclusion, perspectives}
In this paper we studied the convergence properties of the stochastic Ravine optimization algorithm. We verified the intuition provided by recent analysis from the dynamics systems perspective showing that the Ravine and Nesterov accelerated gradient methods behave similarly, with identical convergence properties. Specifically, we showed that the same asymptotic guarantees as well as convergence rates apply with respect to function values, gradients and convergence of the iterates. 




\appendix

\section{Auxiliary lemmas}
We here collect some important results that play a crucial role in the convergence analysis of \ref{eq:NAGalphakstoch}.

\begin{lemma}[Convergence of non-negative almost supermartingales~\cite{RobbinsSiegmund}]\label{lem:RS}
Given a filtration $\RFilt=\seq{\Rfilt_k}$ and the sequences of real-valued random variables $\seq{r_k}\in\ell_+\pa{\RFilt}$, $\seq{a_k}\in\ell_+\pa{\RFilt}$, and $\seq{z_k}\in\ell_+^1\pa{\RFilt}$ satisfying, for each $k \in \N$
\[
\EX{r_{k+1}}{\Rfilt_k} - r_k \leq - a_k + z_k \; \Pas
\]
it holds that $\seq{a_k}\in\ell^1_+\pa{\RFilt}$ and $\seq{r_k}$ converges $\Pas$ to a random variable valued in $[0,+\infty[$.
\end{lemma}

The following lemma is a consequence of Lemma~\ref{lem:RS}; see also the discussion in \cite[Section~3]{RobbinsSiegmund}.
\begin{lemma}\label{lem:RSsum}
Given a filtration $\RFilt=\seq{\Rfilt_k}$, let the sequence of random variables $\seq{\varepsilon_k} \in \ell_+(\RFilt)$ such that $\seq{\pa{\EX{\varepsilon_k^2}{\Rfilt_{k-1}}}^{1/2}} \in \ell_+^1\pa{\RFilt}$. Then
\[
\sum_{k \in \N} \varepsilon_k < +\infty \quad \Pas .
\]
\end{lemma}

\begin{proof}
Let $\zeta_k = \varepsilon_k - \EX{\varepsilon_k}{\Rfilt_{k-1}}$ and $r_k = \pa{\sum_{i=1}^k \zeta_i}^2$. We obviously have $\EX{\zeta_{k+1}}{\Rfilt_k} = 0$. Thus
\begin{align*}
\EX{r_{k+1}}{\Rfilt_k} 
&= \pa{\sum_{i=1}^k \zeta_i}^2 + \sum_{i=1}^k \zeta_i \EX{\zeta_{k+1}}{\Rfilt_k} + \EX{\zeta_{k+1}^2}{\Rfilt_k} \\
&= r_k + \EX{\zeta_{k+1}^2}{\Rfilt_k} 
= r_k + \VarX{\varepsilon_{k+1}^2}{\Rfilt_k} 
\leq r_k + \EX{\varepsilon_{k+1}^2}{\Rfilt_k} .
\end{align*}
It is easy to see that $\seq{\pa{\EX{\varepsilon_k^2}{\Rfilt_{k-1}}}^{1/2}} \in \ell_+^1\pa{\RFilt}$ implies $\seq{\EX{\varepsilon_k^2}{\Rfilt_{k-1}}} \in \ell_+^1\pa{\RFilt}$, and we can apply Lemma~\eqref{lem:RS} to get that 
\[
\lim_{k \to +\infty} r_k
\]
exists and is finite $\Pas$. Using Jensen's inequality we have
\[
0 \leq \sum_{i=1}^{k} \varepsilon_i  = \sum_{i=1}^{k} \zeta_i + \sum_{i=1}^{k} \EX{\varepsilon_i}{\Rfilt_{i-1}} \leq r_k^{1/2} + \sum_{i=1}^{k} \pa{\EX{\varepsilon_i^2}{\Rfilt_{i-1}}}^{1/2} .
\]
Passing to the limit using that $\seq{\pa{\EX{\varepsilon_k^2}{\Rfilt_{k-1}}}^{1/2}} \in \ell_+^1\pa{\RFilt}$ proves the claim. \qed
\end{proof}

%

\begin{lemma}[Extended descent lemma]\label{ext_descent_lemma}
Let  $g: \cH \to \R$ be  a  convex function whose gradient is $L$-Lipschitz continuous. Let $s \in ]0,1/L]$. Then for all $(x,y) \in \cH^2$, we have
\begin{equation}\label{eq:extdesclem}
g(y - s \nabla g (y)) \leq g(x) + \dotp{\nabla g(y)}{y-x} - \frac{s}{2} \norm{\nabla g(y)}^2 -\frac{s}{2} \norm{\nabla g (x) - \nabla g(y)}^2 .
\end{equation}
\end{lemma}
See e.g. \cite[Lemma~1]{ACFR}


\begin{thebibliography}{10}


















 



\bibitem{AD15} {\sc J.-F. Aujol and C. Dossal}, {\it Stability of over-relaxations for the forward-backward algorithm, application to FISTA}, SIAM J. Optim., 25 (2015), 2408--2433.

 

 

 

\bibitem{AC1} {\sc H. Attouch, A.  Cabot}, {\it Asymptotic stabilization of inertial gradient dynamics with time-dependent viscosity},  J. Differential Equations, 263 (2017), pp. 5412-5458.

 
\bibitem{AC2} {\sc H. Attouch, A.  Cabot}, {\it Convergence rates of inertial forward-backward algorithms},  SIAM J. Optim., 28 (1) (2018), 849--874.


 
 
\bibitem{AC2R-JOTA} {\sc H. Attouch, A.  Cabot, Z. Chbani, H. Riahi},
\textit{Accelerated forward-backward algorithms with perturbations. Application to Tikhonov regularization},
 JOTA, 179 (2018), No.1,   pp. 1-36.



  
 
\bibitem{ACFR}  {\sc H. Attouch, Z. Chbani, J. Fadili, H. Riahi}, {\it First-order algorithms via inertial systems with Hessian driven damping},  Math. Program., (2020)
https://doi.org/10.1007/s10107-020-01591-1,  preprint available at arXiv:2107.05943v1 [math.OC] 13 Jul 2021.

 
   

\bibitem{ACPR} {\sc H. Attouch,  Z. Chbani, J. Peypouquet, P. Redont},  {\it Fast convergence of inertial dynamics and algorithms with asymptotic vanishing viscosity}, Math. Program. Ser. B \, 168 (2018),  123--175.
  

{\it Rate of convergence  of the Nesterov accelerated gradient method  in the subcritical case  $\alpha \leq 3$}, 
 ESAIM-COCV,  25 (2019), Article Number 2, https://doi.org/10.1051/cocv/2017083

\bibitem{acz} {\sc H. Attouch and M.-O. Czarnecki},
{\it Asymptotic control and stabilization of nonlinear oscillators with non-isolated equilibria},
J. Differential Equations, 179 (2002), 278--310.
 
\bibitem{AF} {\sc H. Attouch,  J. Fadili},
{\it From the Ravine method to the Nesterov method and vice
versa: a dynamic system perspective}, SIAM J. on Optim., 2023.
 
\bibitem{AFK} {\sc H. Attouch,  J. Fadili, V. Kungurtsev},
{\it On the effect of perturbations, errors in first-order optimization
methods with inertia and Hessian driven damping}, arXiv:2106.16159v1 [math.OC] 30 Jun 2021.

   


  
\bibitem{AP} {\sc H. Attouch, J. Peypouquet},
 {\it The rate of convergence of Nesterov's accelerated forward-backward method is actually faster than $1/k^2$}, SIAM J. Optim., 26(3) (2016),   1824--1834.
 






\bibitem{AR}{\sc M. Assran, and M. Rabbat}, {\it On the convergence of Nesterov's accelerated gradient method in stochastic settings.} Proceedings of the 37th International Conference on Machine Learning, (2020), pp. 410--420





\bibitem{AZ}{\sc Z Allen-Zhu} {\it Katyusha: The first direct acceleration of stochastic gradient methods}, Journal of Machine Learning Research, 18.221 (2018), pp.~1--51.





















\bibitem{DT}{\sc Y. Drori, M. Teboulle}, {\it Performance of first-order methods for smooth convex minimization: a novel approach}, Mathematical Programming, 145 (2014), pp. 451--482.


\bibitem{Driggs22}{\sc D. Driggs, M. Ehrhardt, and C.-B. Sch{\"o}nlieb}, {\it Accelerating variance-reduced stochastic gradient methods}, Mathematical Programming, 191 (2022), pp. 671?-715.

\bibitem{Frostig2015} {\sc R. Frostig, R. Ge, S. Kakade, and A. Sidford},
{\it Un-regularizing: approximate proximal point and faster stochastic algorithms for empirical risk minimization}, In ICML, 37 (2015), 2540--2548.


\bibitem{Gadat2018} {\sc S. Gadat, F. Panloup, and S. Saadane}, 
{\it Stochastic heavy ball}, Electronic Journal of Statistics, 12 (2018), 461--529.
    
\bibitem{GT} {\sc   I.M. Gelfand, M.  Tsetlin}, {\it Printszip nelokalnogo poiska v sistemah avtomatich}, Optimizatsii, Dokl. AN SSSR, 137 (1961), pp. 295--298 (in Russian).
 









\bibitem{HJ1} {\sc A. Haraux and M. A. Jendoubi},
{\it On a second order dissipative ode in Hilbert space with an integrable source term},
Acta Math. Sci., 32 (2012), 155--163.





\bibitem{Jain2018} {\sc P. Jain, S. M. Kakade, R. Kidambi, P. Netrapalli and A. Sidford},
{\it Accelerating stochastic gradient descent for least squares regression}, 
In Proceedings of the 31st Conference On Learning Theory, 75 (2018), 545--604.

\bibitem{KF} {\sc D. Kim, J.A. Fessler}, {\it Optimized first-order methods for smooth convex minimization},  Math. Program. 159(1) (2016), 81--107.


\bibitem{Laborde2020} {\sc M. Laborde and A. Oberman}, 
{\it A Lyapunov analysis for accelerated gradient methods: from deterministic to stochastic case}, 
In Proceedings of the 23rd International Conference on Artificial Intelligence and Statistics (AISTATS), 108 (2020).



\bibitem{LanBook}{\sc G. Lan}, {\it First-order and Stochastic Optimization Methods for Machine Learning}, Springer Series in the Data Sciences, Springer, 2020.

\bibitem{Lin2015} {\sc H. Lin, J. Mairal, and Z. Harchaoui}, 
{\it A universal catalyst for first-order optimization}, In Advances in Neural Information Processing Systems, 28 (2015), 3384--3392.


\bibitem{Loizou2020}, {\sc N. Loizou and P. Richt{\'a}rik},
{\it Momentum and stochastic momentum for stochastic gradient, Newton, proximal point and subspace descent methods}, Comput. Optim. Appl. 77 (2020), 653--710. 






\bibitem{Nest1}{\sc  Y. Nesterov}, {\it A method of solving a convex programming problem with convergence rate}
$O(1/k^2)$, Soviet Mathematics Doklady,  27  (1983),  372--376.

\bibitem{Nest2}{\sc  Y. Nesterov}, {\it  Introductory lectures on convex optimization: A basic course}, volume 87 of Applied Optimization. Kluwer Academic Publishers, Boston, MA, 2004.




\bibitem{PPR}{\sc  C. Park, J. Park, E. K. Ryu}, {\it Factor-$\sqrt{2}$  Acceleration of Accelerated Gradient Methods}, arXiv:2102.07366 [math.OC], 2021.





\bibitem{Polyak_0} {\sc B.T. Polyak}, {\it Accelerated gradient methods revisited}, Workshop Variational Analysis and Applications, August 28-September 5, 2018, Erice.

\bibitem{Polyak_2}{\sc B.T. Polyak}, {\it Introduction to optimization}. New York: Optimization Software. (1987).

\bibitem{Polyak_1}{\sc  B. Polyak},
Some methods of speeding up the convergence of iteration methods, USSR Computational Mathematics and Mathematical Physics, 4 (1964), pp. 1--17.

\bibitem{RobbinsSiegmund}{\sc  H. Robbins and D. Siegmund},
{\it A Convergence Theorem for Non Negative Almost Supermartingales and Some Applications}. In Herbert Robbins Selected Papers, Springer New York, pp. 111--135, 1985.





\bibitem{SLB} {\sc M. Schmidt, N. Le Roux and F. Bach},
{\it Convergence Rates of Inexact Proximal-Gradient Methods for Convex Optimization},
In Advances in Neural Information Processing Systems, Dec 2011, Grenada, Spain, HAL inria-00618152v3.

\bibitem{SDJS}{\sc B. Shi, S.  S. Du,  M. I. Jordan,  W. J. Su},
{\it Understanding the acceleration phenomenon via high-resolution differential equations},   Math. Program. (2021). https://doi.org/10.1007/s10107-021-01681-8.

\bibitem{Sebbouh2021} {\sc O. Sebbouh, R. M. Gower and A. Defazio},
{\it Almost sure convergence rates for stochastic gradient descent and stochastic heavy ball}, 
In Proceedings of 34th Conference on Learning Theory, 134 (2021), 3935--3971.

\bibitem{SBC}{\sc W. J. Su,  S. Boyd,  E. J. Cand\`es}, {\it A differential equation for modeling Nesterov's accelerated gradient method: theory and insights}, In Advances in Neural Information Processing Systems, 27 (2014), 2510--2518. 


\bibitem{VSBV}{\sc S. Villa, S. Salzo, L. Baldassarres, A. Verri}, {\it Accelerated and inexact forward-backward}, SIAM J. Optim.,  23  (2013), No. 3, 1607--1633.

\bibitem{Yang2016} {\sc T. Yang, Q. Lin, and Z. Li}, 
{\it Unified convergence analysis of stochastic momentum methods for convex and non-convex optimization}, arXiv:1604.03257 (2016).

\end{thebibliography}
\end{document}